\documentclass[reqno]{amsart}
\usepackage{latexsym}
\usepackage{amssymb,amsthm,amsmath}
\usepackage[dvips]{graphicx}
\usepackage{color}

\allowdisplaybreaks 

\calclayout

\theoremstyle{plain}
\newtheorem{theorem}{Theorem}
\newtheorem{lemma}[theorem]{Lemma}
\newtheorem{corollary}[theorem]{Corollary}

\theoremstyle{definition}
\newtheorem{definition}[theorem]{Definition}
\theoremstyle{remark}

\DeclareMathOperator{\esssup}{esssup}

\DeclareMathOperator{\Div }{div}

\newcommand{\p}{\partial}
\newcommand{\n}{\triangledown}
\newcommand{\ov}{\overline}
\DeclareMathOperator{\diver}{div}

\newcommand{\mb}{\mathbb}
\DeclareMathOperator{\sign}{sign}

\DeclareMathOperator{\integra}{\int\limits_0^T \int\limits_{M^2} \int\limits_{\mb{R}^2}}
\DeclareMathOperator{\integral}{\int\limits_{M^2} \int\limits_{\mb{R}^2}}

\def\pa{\partial}
\def\cal{\mathcal}

\def\R{{\mathbb R}}
\def\N{{\mathbb N}}
\def\Z{{\mathbb Z}}

\def\eps{\varepsilon}

\def\mx{{\bf x}}
\def\tmx{\tilde{{\bf x}}} 
\def\tmy{\tilde{{\bf y}}} 
\def\mz{{\bf z}}

\def\my{{\bf y}}
\def\mz{{\bf z}}

\def\mff{{\mathfrak f}}
\def\mzt {{\zeta}}

\begin{document}

\title[Stochastic CL on Riemannian manifolds]{Well-posedness for stochastic scalar conservation laws on Riemannian manifolds}

\author{N.\ Konatar}\address{Nikola Konatar, University of Montenegro, 
Faculty of Mathematics}\email{nikola.k@ac.me}
\author{D.\ Mitrovic}\address{Darko Mitrovic,
Faculty of Mathematics, Oscar Morgenstern platz 1,
1090 Vienna, Austria}\email{darkom@ac.me}
\author{E.Nigsch}\address{Eduard Nigsch,
Faculty of Mathematics, University of
Vienna, Oscar Morgenstern Pl.1, Vienna, Austria }\email{eduard.nigsch@univie.ac.at}



\begin{abstract}
We consider the scalar conservation law with stochastic forcing
$$
d u +\mathrm{div}_g {\mathfrak f}(\mx,u) dt= \Phi(\mx,u) dW_t, 
\ \ {\bf x} \in M, \ \ t\geq 0
$$ on a smooth compact Riemannian manifold $(M,g)$ where $W_t$ is the Wiener process and ${\bf x}\mapsto {\mathfrak f}(\mx,\xi)$ is a vector field on $M$ for each $\xi\in \R$. We introduce admissibility conditions, derive the kinetic formulation and use it to prove well posedness.
\end{abstract}
\maketitle

\section{Introduction}

We consider the Cauchy problem for a stochastic scalar conservation law of the form
\begin{align}
\label{main-eq}
du +\Div_g \mff(\mx,u) dt &=\Phi(\mx,u) dW_t, \ \ \mx \in M, \ \ t\geq 0\\
\label{ic}
u|_{t=0} &=u_0(\mx) \in L^\infty(M)
\end{align} on a smooth, compact, $d$-dimensional (Hausdorff) Riemannian manifold $(M,g)$. The object $W$ is the Wiener process which can be finite or infinite dimensional which does not affect the essence of the proofs. Therefore, we shall assume that we work with one-dimensional Wiener process defined on the stochastic basis $(\Omega, {\cal F}, ({\cal F}_t), {\bf P})$. 

We will assume that 

\begin{itemize}

\item the flux $\mff$ satisfies the geometry compatibility conditions and a decay property as follows respectively:
\begin{align}
\label{geomcomp}
&\Div_g \mff({\mx},\xi)=0 \ \ \text{for every $\xi\in \R$}\\
\label{decay}
&\sup\limits_{\lambda \in \R}| \mff(\cdot,\lambda) |\in L^2(M)  \ \ {\rm and} \ \ \| \mff (\cdot,\lambda)\|_{L^\infty(M)}\leq C(1+|\lambda|);
\end{align}

\item the function $\Phi$ is continuously differentiable and it decays to zero at infinity i.e. $\Phi \in C^1_0(M\times \R)$, and 
\begin{equation}
\label{assump-Phi}
\sup\limits_{\lambda \in \R}|\Phi(\cdot,\lambda) \lambda| \in L^1(M).
\end{equation}

\end{itemize}


Nowadays, we are witnessing a rapid development of stochastic conservation laws and related equations. The rising interest to this field of research is motivated by concrete applications in biology, porous media, finances (see e.g. randomly chosen \cite{Allen, kenneth, TNL} and references therein) and, in general, any realistic situation in which we cannot determine parameters precisely (i.e. the coefficients of the equations governing the process).


Moreover, such equations have rich mathematical structure and therefore, they are very interesting and challenging from the mathematical point of view. We have numerous results in different directions beginning with the stochastic conservation laws \cite{CDK, DV, FG, FN, Hof, Kim03, WKMS}, then velocity averaging results for stochastic transport equations \cite{DV-va, LS}, stochastic degenerate parabolic equations \cite{GH, china}. We remark that latter list of references is far from complete.  Here, we aim to expand the theory of stochastic scalar conservation laws to manifolds. As for the stochastic PDEs on manifolds, we mention \cite{BO} where the wave equation was considered.

Let us now briefly recall the meaning of the divergence on a manifold.  We suppose that the map $(\mx,\xi)\mapsto \mff(\mx,\xi)$, 
$M\times \R \to TM$ is $C^1$ and that, for every $ \xi\in \R$,
$\mx \mapsto \mff(\mx,\xi)\in \mathfrak{X}(M)$ (the space of vector fields on $M$).

In local coordinates, we write
\begin{align*}
\mff(\mx,\xi)=(f^1(\mx,\xi),\dots,f^d(\mx,\xi)).
\end{align*}  
The divergence operator appearing in the equation is to be formed with respect to the metric,
so in local coordinates we have (cf.\ \eqref{divx} below):
\begin{equation}
\label{div-g}
\Div_g \mff({\mx},u)=\Div_g \big(\mx \mapsto \mff({\mx},u(t,\mx))\big)=\frac{\pa}{\pa x_k} (f^k({\mx},u(t,\mx))+\Gamma^j_{kj}(\mx) f^k({\mx},u(t,\mx))
\end{equation} where the $\Gamma$-terms are the Christoffel symbols of $g$ and the Einstein summation convention is in effect. 

As we can see, the divergence operator on manifolds is more involved than the one in Euclidean setting. Therefore, in order to prove uniqueness, we need to assume \eqref{geomcomp}. Remark that \eqref{geomcomp} are the incompressibility condition from the fluid dynamics point of view.  Let us briefly explain why. Due to conservation of mass of an incompressible fluid, the density in a control volume changes according to the stochastic forcing
\begin{equation}
\label{1}
\frac{D\rho }{Dt}= \Phi(\mx,\rho) \frac{dW_t}{dt}
\end{equation}where $\rho$ is density of the control volume and $\frac{D\rho }{Dt}=\frac{\pa \rho}{\pa t}+\frac{d\mx}{dt} \cdot \nabla \rho$ is the material derivative for the flow velocity $\frac{d\mx}{dt}=(\frac{dx_1}{dt},\dots,\frac{dx_d}{dt})$. If we assume that the function $\rho$ is smooth, we can rewrite equation \eqref{main-eq} in the form
\begin{equation}
\begin{split}
\frac{\pa \rho}{\pa t}+\pa_\xi \big(\mff(\mx,\xi)\big)\big{|}_{\xi=\rho} \cdot \nabla_g \rho+\diver_g \mff(\mx,\xi)\big{|}_{\xi=\rho} =\Phi(\mx,\rho) \frac{dW}{dt}.
\label{2}
\end{split}
\end{equation} Then, taking as usual 
$
\frac{d\mx}{dt}=\pa_\xi \big(\mff(\mx,\xi)\big)\big{|}_{\xi=\rho}
$ and comparing \eqref{2} and \eqref{1}, we arrive at
$$
\Div_g \mff(\mx,\xi)\big{|}_{\xi=\rho}=0,
$$  which immediately gives what is called the geometry compatibility condition.

Since the equation we consider is a nonlinear hyperbolic equation, its solution in general contains discontinuities and we need to pass to the weak solution concept. However, this induces uniqueness issues as one can in general construct several weak solutions satisfying the same initial data. Thus, in order to isolate the physically admissible one, we need to introduce entropy type admissibility conditions \cite{Kru}. We will first derive them locally and then, using the geometry compatibility conditions, we shall show that the conditions hold globally as well.

Having the admissibility conditions, we can derive the kinetic formulation to \eqref{main-eq} (see \eqref{WeakForm3}). We will use it to prove both existence and uniqueness to the considered Cauchy problem. The strategy of proof is adapted from \cite{DV}. We have tried to be as precise and self contained and intuitive as possible. We therefore proved a simple corollary of the It\^o lemma concerning the derivative of the product of two stochastic processes and derive the uniqueness proof first informally, and then also formally.

The paper is organized as follows. In Section \ref{diffprelsec} we introduce notions and notations from differential geometry and stochastic calculus. We then move on to derive the kinetic formulation of \eqref{main-eq} and heuristically show how to get uniqueness to the solution. In Section \ref{uniquenessec}, we formally prove the uniqueness result. Finally, in Section \ref{existencesec} we show existence of the kinetic solution which in turn implies existence of the entropy admissible solution.

\section{Preliminaries from Riemannian geometry and stochastic calculus}\label{diffprelsec}

We shall split the section into two parts. In the first one, we will provide details from differential geometry, and in the second one, we recall necessary results from stochastic calculus.

\subsection{Riemannian geometry}

Our standard references for notions from Riemannian and distributional geometry are \cite{ GKOS, Mar, ON83, Pet06}. 
As before,
$(M,g)$ will be a $d$-dimensional Riemannian manifold.  If $v$ is a distributional vector field on $M$  
then its gradient $\nabla v$ is the
vector field metrically equivalent to the exterior derivative $dv$ of $v$: $\langle \nabla v, X\rangle = dv(X) = X(v)$
for any $X\in \mathfrak{X}(M)$. 
In local coordinates,
\begin{equation}
\label{nabla}
\nabla v = g^{ij} \frac{\pa v}{\pa x^i} \pa_{j},
\end{equation} 
with $g^{ij}$ the inverse matrix to $g_{ij}=\langle \pa_{x^i},\pa_{x^j}\rangle$.

As for the Laplace-Beltrami operator $\Delta_g$ on $M$, we have for a function $f\in C^2(M)$ in terms of local coordinates
$$
\Delta_g f=\nabla_g^2 f=\frac{1}{\sqrt{|g|}} \pa_i\left( \sqrt{|g|} g^{ij} \pa_j f \right)
$$

Finally, the divergence operator on $M$ is locally defined via the Christofel symbols for a $C^1$ vector field on $X\in \mathcal{T}^1_0=\mathfrak{X}(M)$ 
with local representation $X=X^i\frac{\pa}{\pa x^i}$:

\begin{equation}\label{divx}
\Div X = \frac{\pa X^k}{\pa x^k} + \Gamma^j_{kj}X^k.
\end{equation}


To proceed, we shall need basic notions from the Sobolev spaces on manifolds. 

Since $M$ is a compact manifold, we can define for a fixed $k\in \N$ (keeping in mind the Poincare inequality)
$$
f\in H^k(M) \ \ \Leftrightarrow \ \ \|\nabla_g^k f\|_{L^2(M)} <\infty.
$$ As for for the Sobolev spaces with negative indexes, we have 
$$
f\in H^{-k}(M) \ \ {\rm if} \ \ \exists F\in H^k(M) \ \ \text{such that} \ \ \Delta^{2k}F=f
$$and we define
\begin{equation}
\label{H-1}
\|f\|_{H^{-k}(M)}=\|F\|_{H^k(M)}.
\end{equation} The spaces $H^k(M)$, $k\in \Z$, are Hilbert spaces and we denote by $\{ e_k \}_{k\in \N}$ the orthogonal basis in $L^2(M)$ which is given as the set of eigenfunctions corresponding to the Laplace-Beltrami operator \cite{}:
 $$
 \Delta_g e_k(\mx)=\lambda_k e_k(\mx).
 $$ At the same time, the set $\{ e_k \}_{k\in \N}$ is the basis in $H^s(M)$, $s \in \Z$, according to the density arguments.

Notice that if we have a function $g\in H^{k}(M)$ and we rewrite it in the basis $\{e_k/\|e_k\|_{H^k(M)}\}$:
\begin{equation}
\label{g}
g(\mx)=\sum\limits_{k=1}^\infty g_k e_k(\mx)/\|e_k\|_{H^k(M)}
\end{equation} then
\begin{equation}
\label{gk}
g_k=\int_M g(\mx) \frac{e_k(\mx)}{\|e_k\|_{H^k(M)}}d\mx
\end{equation} which is easily seen by multiplying \eqref{g} by $e_k/\|e_k\|_{H^k(M)}$, integrating the result over $M$ and using the orthogonality of $\{e_k/\|e_k\|_{H^k(M)}\}$. Moreover, 
\begin{equation}
\label{l2}
\|g\|_{H^{k}(M)}=\sum\limits_{k=1}^\infty g_k^2.
\end{equation} It is not difficult to notice that according to the definition of $e_k$ and \eqref{H-1}, we have
\begin{equation}
\label{lambda-k}
\|e_k\|_{L^2(M)}=\sqrt{\lambda_k} \|e_k\|_{H^{-1}(M)}.
\end{equation}

Let us now recall basic notions from stochastic calculus.

\subsection{Stochastic calculus}

What we essentially need from the stochastic calculus is the It\^o lemma and some of its corollaries. To this end, let $X_t$ be a stochastic process satisfying the following stochastic differential equation:
\begin{equation}
\label{StocProc}
dX_t = \mu_1 dt + \sigma_1 dW_t.  
\end{equation} We remark here that the latter equation is actually an informal way of expressing the integral equality
\begin{equation}
\label{integral}
X_{t_0+s}-X_{t_0}=\int_{t_0}^{t_0+s}\mu_1 dt + \int_{t_0}^{t_0+s} \sigma_1 dW_t, \ \ \forall t_0,s >0.  
\end{equation}

By It\^o's lemma, for each twice differentiable scalar function $f=f(t,z)$ the equation 
\begin{align}
\label{ItoFormula}
df(X_t)&=\left( \frac{\p f}{\p t} + \mu_1 \frac{\p f}{\p z} + \frac{\sigma^2_1}{2} \frac{\p^2 f}{\p z^2} \right) dt + \sigma_1 \frac{\p f}{\p z} dW_t
\end{align}
holds.

By taking $f(t,X_t)=X^2_t$, we get
\begin{equation}
\label{ItoFormula1}
d X^2_t= 2 \mu_1 X_t dt + \sigma^2_1 dt + 2 \sigma_1 X_t dW_t. 
\end{equation}

Notice that $2 \mu_1 X_t dt + 2 \sigma_1 X_t dW=2X_t( \mu_1  dt + \sigma_1 dW)=2X_t dX_t$, so \eqref{ItoFormula1} becomes 
\begin{equation}
\label{ItoFormula2}
d X^2_t= 2X_t dX_t + \sigma^2_1 dt.
\end{equation}

Similarly, if $Y_t$ is a stochastic process satisfying the stochastic differential equation
\begin{equation}
\label{StocProc1}
dY_t = \mu_2 dt + \sigma_2 dW_t  
\end{equation}
then 
\begin{align}
\label{ItoFormula3}
d Y^2_t &= 2Y_t dY_t + \sigma^2_2 dt, \\ 
\label{ItoFormula4}
d(X_t+Y_t)^2 &= 2 (X_t+Y_t)d(X_t+Y_t) + (\sigma_1 + \sigma_2)^2 dt.
\end{align}

The left-hand side of \eqref{ItoFormula4} is
\begin{multline}
\label{ItoFormula4L}
d(X_t+Y_t)^2 = d(X^2_t + 2 X_t Y_t + Y^2_t)= d X^2_t + 2 d (X_t Y_t) + d Y^2_t \\ = 2X_t dX_t + \sigma^2_1 dt + 2 d (X_t Y_t) + 2Y_t dY_t + \sigma^2_2 dt,
\end{multline}
and the right side is
\begin{multline}
\label{ItoFormula4R}
2 (X_t+Y_t)d(X_t+Y_t) + (\sigma_1 + \sigma_2)^2 dt  \\ = 2 X_t d X_t + 2 X_t d Y_t + 2 Y_t d X_t + 2 Y_t d Y_t + \sigma^2_1 dt +2 \sigma_1 \sigma_2 dt  + \sigma^2_2 dt.
\end{multline}

By annuling the same terms on the left and right side respectively, and dividing the equation by 2, we get
\begin{equation}
\label{ItoFormulaFinal}
d (X_t Y_t) = X_t d Y_t + Y_t d X_t + \sigma_1 \sigma_2 dt.
\end{equation}

Let us finally recall the It\^o isometry. It holds
\begin{equation*}
E\left[ \left( \int_0^TX_t dW_t \right)^2\right]=E\left[ \int_0^T X_t^2 dt \right].
\end{equation*}

\section{Entropy admissibility and kinetic formulation}


In order to derive the admissibility conditions, we shall, as usual, start with the parabolic approximation to \eqref{main-eq}
\begin{equation}
\label{SCL}
du_\eps+\diver_g (\mff (\mx,u_\eps)) dt=\Phi (\mx,u_\eps) dW_t+\eps \Delta_g u_\eps dt ,~ \mx \in M, \, t\in (0,T)
\end{equation}  where, as before, $\mff =\mff (\mx,\lambda) \in C^1(M\times \R)$ and $(M,g)$ is a $d$-dimensional Rimannian manifold with the metric $g$. We will assume that $W_t$ is a Wiener process and $\Phi \in C^1_0(M\times \R)$. 

Since we are dealing with the stochastic parabolic equation on a manifold, we cannot say anything about the existence of solution to the appropriate Cauchy problem. However, we shall assume that we can find a smooth solution to \eqref{SCL}, \eqref{ic} and prove later that this indeed holds.

Using the It\^o formula, from \eqref{SCL} we get (here and in the sequel, we will set $\mff'(\mx,\xi)=\pa_\xi f(\mx,\xi)$):
\begin{equation}
\label{ISCL}
\begin{split}
d\theta(u_\eps)&=\Big(-\theta'(u_\eps) \mff'(\mx,u_\eps) \cdot \nabla_g u_\eps +\theta'(u_\eps) \Div_g \mff (\mx,\rho)\big|_{\rho=u_\eps}\\&+ \eps \Delta_g \theta(u_\eps)-{\eps}\theta''(u_\eps) |\nabla_g u_\eps|^2  +\frac{\Phi^2(\mx,u_\eps)}{2}\theta'' (u_\eps)\Big)dt+\Phi(\mx,u_\eps)\theta'(u_\eps)dW_t
\end{split}
\end{equation}
for all twice differentiable scalar functions $\theta$.

Using the standard approximation procedure and taking into account convexity of the function $\theta(u)=|u-\xi|_+=\begin{cases}
u-\xi, &u\geq \xi\\
0, & else
\end{cases}$, we know that we can safely plug it into \eqref{ISCL}. After letting $\eps \to 0$ and assuming that $E(|u_\eps(t,\mx)-u(t,\mx)|) \to 0$ as $\eps\to 0$, we get the following distributional inequality:

\begin{equation}
\label{WeakForm}
\begin{split}
{d | u-\xi |_+}&\leq - \mff'(\mx, u) \n_g u \sign_+(u-\xi)dt +\theta'(u) \Div_g \mff (\mx,\rho)\big|_{\rho=u}dt\\&+ \frac{\Phi^2 (\mx,u)}{2} \delta (u-\xi)dt + \Phi (\mx,u) \sign_+ (u-\xi) {dW_t}.
\end{split}
\end{equation}

Taking into account the geometry compatibility condition \eqref{geomcomp}, we have
\begin{equation}
\label{aux1}
\begin{split}
&\mff'(\mx,u)\cdot (\n_g u) \sign_+ (u-\xi) = \diver_g \left(\sign_+ (u-\xi)(\mff (\mx,u)-\mff (\mx,\xi))\right)\\
&+ \sign_+(u-\xi)\diver_g \mff (\mx,\xi)=\diver_g \left(\sign_+ (u-\xi)(\mff (\mx,u)-\mff (\mx,\xi))\right),
\end{split}
\end{equation} and using the Schwartz lemma on non-negative distributions, we conclude that there exists a non-negative stochastic kinetic measure $m$ (to be precised later) such that the equation \eqref{WeakForm} can be written as

\begin{equation}
\label{WeakForm1}
\begin{split}
d |u-\xi|_+=-\diver_g (\sign_+ (u-\xi) (\mff (\mx,u)-\mff (\mx,\xi)))dt+ \frac{\Phi^2 (\mx, u)}{2} \delta (u-\xi) dt& \\+ \Phi (\mx,u) \sign_+(u-\xi) dW_t {-} dm(t,\mx,\xi)dt.&
\end{split}
\end{equation}

Next, we find the partial derivative of the expression given in \eqref{WeakForm1} with respect to $\xi$ to get 

\begin{equation}
\label{WeakForm2}
\begin{split}
d \p_\xi |u-\xi|_+=-\diver_g (-\mff'(\mx,\xi) \sign_+(u-\xi))dt +
\p_\xi \left(\frac{\Phi^2 (\mx,u)}{2} \delta (u-\xi) \right)dt&\\
 + \p_\xi (\Phi (\mx,u) \sign_+(u-\xi) dW_t)-\p_\xi dm.&
\end{split}
\end{equation}
 
Introducing $h(t,x,\xi)=-\p_\xi |u-\xi|_+=\sign_+(u-\xi)$ into \eqref{WeakForm2} gives 

\begin{equation}
\label{WeakForm3}
d h+\diver_g (\mff'(\mx,\xi) h)dt = - \p_\xi \left(\frac{\Phi^2 (\mx, u)}{2} \delta (u-\xi)\right)dt - \p_\xi (\Phi (\mx,u) h dW_t)+\p_\xi dm.
\end{equation}

Notice that 
\begin{equation}
\label{aux2}
\begin{split}
&\p_\xi (\Phi (\mx,u) h d W_t) = \p_\xi (\Phi (\mx,u) \sign_+ (u-\xi)) dW_t= -\Phi (\mx,u) \delta (u-\xi) dW_t\\&= -\Phi(\mx,\xi) \delta (u-\xi) dW_t.
\end{split}
\end{equation}

Using $\frac{\Phi^2 (\mx, u)}{2} \delta (u-\xi)=\frac{\Phi^2 (\mx,\xi)}{2} \delta (u-\xi)$ and \eqref{aux2}, and denoting the measure $-\p_\xi h=\delta (u-\xi)$ by $\nu_{(t,\mx)}(\xi)$, we finally get the weak form of our equation:

\begin{equation}
\label{WeakFormFinal}
 dh+\diver_g (\mff'(\mx, \xi)  h)dt=  - \p_\xi \left(\frac{\Phi^2 (\mx,\xi)}{2} \nu_{(t,\mx)}(\xi) \right)dt + \Phi(\mx,\xi) \nu_{(t,\mx)}(\xi) W_t +\p_\xi dm\,.
\end{equation} We shall call the latter equation {\em the kinetic formulation} of \eqref{main-eq}.

It is important to notice that the function $\ov{h}=1-h$ satisfies 
\begin{equation}
\label{WeakFormConjugate}
d \ov{h}+\diver_g (\mff'(\mx,\xi)  \ov{h}) dt = \p_\xi \left(\frac{\Phi^2 (\mx,\xi)}{2} \nu_{(t,\mx)}(\xi) \right) dt - \Phi(\mx,\xi) \nu_{(t,\mx)}(\xi) dW_t -\p_\xi dm.
\end{equation}

We can now introduce a definition of an admissible solution. Let us first introduce what we meant under the stochastic measure here.

\begin{definition}
\label{StochM}
We say that a mapping $m$ from $\Omega$ into the space of Radon measures on $[0,T]\times M \times \R$ is a stochastic kinetic measure if:
\begin{itemize}


\item for every $\phi \in C_0([0,T]\times M\times \R)$  the action $\langle m, \phi \rangle$ defines a ${\bf P}$-measurable function
$$
\langle m, \phi \rangle : \Omega \to R;
$$

\item for every $\phi\in C_0(M\times \R)$, the process
$$
t\mapsto \int_{[0,t]\times M \times \R} \phi(\mx,\xi) dm(s,\mx,\xi)
$$
\end{itemize}is predictable.
\end{definition}

\begin{definition}
\label{admissibility}
The measurable function $u: [0,T]\times M \times \Omega \to \R $ almost surely continuous with respect to time in the sense that $u(\cdot,\cdot,\omega) \in C(\R^+;{\cal D}'(M))$ for ${\bf P}$-almost every $\omega\in \Omega$ is an admissible stochastic solution to \eqref{main-eq}, \eqref{ic} if 

\begin{itemize}


\item there exists $C_2>0$ such that $E(\esssup\limits_{t\in [0,T]} \|u(t)\|_{L^2(M)})\leq C_2$;

\item the kinetic function $h={\rm sign}_+(u-\xi)$ satisfies \eqref{WeakForm1} with the initial conditions $h(0,\mx,\xi)={\rm sign}_+(u_0(\mx)-\xi) $ in the sense of weak traces and $\ov{h}$ satisfies \eqref{WeakFormConjugate} with the initial conditions $\ov{h}(0,\mx,\xi)=1-{\rm sign}_+(u_0(\mx)-\xi) $ in the sense of weak traces.
\end{itemize}
\end{definition}

We shall also need a notion of the kinetic solution.

\begin{definition}
\label{kineticsol}
{A measurable function} $h=h(t,\mx,\xi,\omega)$, $(t,\mx,\xi,\omega)\in \R^+\times M \times \R\times \Omega$, bounded between zero and one and non-strictly decreasing with respect to $\xi \in \R$ is the stochastic kinetic solution to \eqref{main-eq}, \eqref{ic} if 

\begin{itemize}


\item {There exists a stochastic kinetic measure $m$ such that} $h$ satisfies \eqref{WeakFormFinal} and the initial conditions $h(0,\mx,\xi)={\rm sign}_+(u_0(\mx)-\xi)$ in the sense of weak traces.

\end{itemize}
\end{definition}

Clearly, if we have the admissible solution to \eqref{main-eq}, \eqref{ic} then we have the kinetic solution as well. Interestingly, vice versa also holds which we will show in the next sections.

\section{Informal uniqueness proof -- doubling of variables}

In this section, we shall informally show how to get uniqueness (which paves the way for the existence as well). Formal proof does not essentially differ from the procedure given in this section but one needs to introduce several smoothing procedures which significantly complicates following steps of the proof. We also remark that, in order to simplify the notation, we will denote by $d\mx$ the measure on the manifold instead of usual $d\gamma(\mx)$.

Let $h^1(t,\mx,\xi)$ and $h^2(t,\my,\zeta)$ be two different kinetic solutions to \eqref{main-eq}, \eqref{ic} (see Definition \ref{kineticsol}). Then
\begin{align}
\label{WeakFormN1}
d h^1+\diver_g (\mff'(\mx,\xi) h^1) dt &=  - \p_\xi \left(\frac{\Phi^2 (\mx, \xi)}{2} \nu^1 \right) dt + \Phi(\mx,\xi) \nu^1 dW_t +\p_\xi dm_1, \\
\label{WeakFormN2}
d \ov{h^2}+\diver_g (\mff'(\my,\zeta)  \ov{h^2}) dt &=  \p_\zeta \left(\frac{\Phi^2 (\my,\zeta)}{2} \nu^2 \right) dt - \Phi(\my, \zeta) \nu^2 dW_t-\p_\zeta dm_2.
\end{align}

By \eqref{ItoFormulaFinal}, the following holds:
\begin{equation}
\label{Doubling1}
d (h^1 \ov{h^2}) = h^1 d \ov{h^2} + \ov{h^2} d h^1 - \Phi(\mx,\xi) \Phi(\my,\zeta) \nu^1 \otimes \nu^2 dt.
\end{equation}

Multiplying \eqref{WeakFormN1} by $\ov{h^2}=\ov{h^2}(t,\my,\zeta)$, \eqref{WeakFormN2} by $h^1=h^1(t,\mx,\xi)$, adding them and using the geometry compatibility conditions \eqref{geomcomp}, yields
\begin{align}
\label{Doubling2}
\nonumber
& \ov{h^2} d h^1 + h^1 d \ov{h^2}+ \ov{h^2} \mff'(\mx,\xi) \cdot \n_{g,\mx} h^1 dt + h^1 \mff'(\my,\zeta) \cdot \n_{g,\my} \ov{h^2} dt \\
\nonumber
& = - \ov{h^2} \p_\xi \left(\frac{\Phi^2 (\mx,\xi)}{2} \nu^1 \right) dt + h^1 \p_\zeta \left(\frac{\Phi^2 (\my,\zeta)}{2} \nu^2 \right) dt + \ov{h^2} \Phi(\mx,\xi) \nu^1 dW_t - h^1 \Phi(\my,\zeta) \nu^2 dW_t\\
& +\ov{h^2} \p_\xi dm_1(t,\mx,\xi)-{h}^1 \p_\zeta dm_2(t,\my,\zeta) dt. 
\end{align}

Inserting \eqref{Doubling1} into \eqref{Doubling2}, we get
\begin{align}
\label{Doubling3}
\nonumber
 & d (h^1 \ov{h^2}) + \Phi(\mx,\xi) \Phi(\my,\zeta) \nu^1 \otimes \nu^2 dt + \ov{h^2} \mff'(\mx,\xi) \cdot \n_{g,\mx} h^1 dt + h^1 \mff'(\my,\zeta) \cdot \n_{g,\my} \ov{h^2} dt \\
\nonumber
 & = - \ov{h^2} \p_\xi \left(\frac{\Phi^2 (\mx,\xi)}{2} \nu^1 \right) dt + h^1 \p_\zeta \left(\frac{\Phi^2 (\my,\zeta)}{2} \nu^2 \right) dt +  (\ov{h^2} \Phi(\mx,\xi) \nu^1 - h^1 \Phi(\my,\zeta) \nu^2 )dW_t \\ 
& + \ov{h^2} \p_\xi dm_1(t,\mx,\xi) dt -{h}^1 \p_\zeta dm_2(t,\my,\zeta) dt. 
\end{align}

We now choose the non-negative test function $\varphi(t,\mx,\my,\xi,\zeta)= \rho(\mx-\my) \psi (\xi-\zeta)$, where $\rho$ and $\psi$ are smooth non-negative functions defined on appropriate Euclidean spaces. Multiplying \eqref{Doubling3} with $\varphi$ and integrating over $(0, T) \times M^2 \times \mathbb{R}^2$ we get 




\begin{align}
\label{Main}
&\integral h^1 (T,\mx,\xi) \ov{h^2} (T,\my,\zeta)\rho(\mx-\my) \psi (\xi-\zeta) d\zeta d\xi d\my d\mx  \\&- \int\limits_{{\rm M}^2}\int\limits_{\R^2} h^1_0 \ov{h^2_0} \rho(\mx-\my) \psi (\xi-\zeta)d\zeta d\xi d\my d\mx  \nonumber \\ 
&+ \integra  \rho(\mx-\my) \psi (\xi-\zeta) \Phi(\mx,\xi) \Phi(\my,\zeta) d\nu_{(t,\my)}^2(\zeta) d\nu_{(t,\mx)}^1(\xi) d\my d\mx  dt
\nonumber \\ 
&+\integra \mff'(\mx,\xi) \cdot \nabla_{g,\mx} h^1 (t,\mx,\xi) \ov{h^2}(t,\my,\zeta)  \rho(\mx-\my) \psi (\xi-\zeta) d\zeta d\xi d\my d\mx dt \nonumber \\ 
&+\integra \mff'(\my,\zeta) \cdot \nabla_{g,\my} \ov{h^2}(t,\my,\zeta) h^1 (t,\mx,\xi) \rho(\mx-\my) \psi (\xi-\zeta) d\zeta d\xi d\my d\mx dt \nonumber \\ 
&= \integra \frac{\Phi^2 (\mx,\xi)}{2} \ov{h^2}(t,\my,\zeta) \rho (\mx-\my) \psi' (\xi-\zeta)    d\nu_{(t,\mx)}^1(\xi) d\zeta d\my d\mx dt \nonumber \\ 
&+ \integra \frac{\Phi^2 (\my,\zeta)}{2} h^1 (t,\mx,\xi) \rho (\mx-\my) \psi' (\xi-\zeta) d\nu_{(t,\my)}^2(\zeta) d\xi  d\my d\mx dt \nonumber \\ 
&+ \integra \rho(\mx-\my) \psi (\xi-\zeta) \ov{h^2}(t,\my,\zeta) \Phi(\mx,\xi)  d\nu_{(t,\mx)}^1(\xi) d\zeta d\my d\mx dW_t \nonumber \\
&-  \integra  \rho(\mx-\my) \psi (\xi-\zeta) h^1 (t,x,\xi) \Phi(\my,\zeta) d\nu_{(t,\my)}^2(\zeta) d\xi d\my d\mx  dW_t \nonumber \\
&-\integra \rho(\mx-\my) \psi' (\xi-\zeta) \ov{h^2}(t,\my,\zeta) d m_1(t,\mx,\xi)  d\zeta d\my \nonumber \\ &-\integra  {h}^1(t,\mx,\xi) \rho(\mx-\my) \psi' (\xi-\zeta) d m_2(t,\my,\zeta)  d\xi d\mx.
\nonumber
\end{align}

By using integration by parts with respect to $\zeta$ and $\xi$ in the first and second and in the last two terms on the right hand side in \eqref{Main}, and using $\p_\xi h^1=-\nu^1$ and $\p_\zeta \ov{h^2}=\nu^2$, we obtain:
\begin{align}
\label{Main1}
&\integral h^1 (T,\mx,\xi) \ov{h^2} (T,\my,\zeta) \rho(\mx-\my) \psi (\xi-\zeta) d\zeta d\xi d\my d\mx \\ &- \int\limits_{{\rm M}^2}\int\limits_{\R^2} h^1_0(\mx,\xi) \ov{h^2_0}(\my,\zeta) \rho(\mx-\my) \psi (\xi-\zeta)d\zeta d\xi d\my d\mx \nonumber \\ 
&+ \integra  \rho(\mx-\my) \psi (\xi-\zeta) \Phi(\mx,\xi) \Phi(\my,\zeta) d\nu_{(t,\my)}^2(\zeta) d\nu_{(t,\mx)}^1(\xi) d\my d\mx dt \nonumber
\\ 
&+\integra \mff'(\mx,\xi) \cdot \nabla_{g,\mx} h^1 (t,\mx,\xi) \ov{h^2}(t,\my,\zeta)  \rho(\mx-\my) \psi (\xi-\zeta) d\zeta d\xi d\my d\mx dt \nonumber \\ 
& +\integra \mff'(\my,\zeta) \cdot \nabla_{g,\my} \ov{h^2}(t,\my,\zeta)  h^1 (t,\mx,\xi)  \rho(\mx-\my) \psi (\xi-\zeta) d\zeta d\xi d\my d\mx dt \nonumber \\ 
& = \integra \frac{\Phi^2 (\mx,\xi)}{2}  \rho (\mx-\my) \psi (\xi-\zeta) d\nu_{(t,\my)}^2(\zeta) d\nu_{(t,\mx)}^1(\xi)  d\my d\mx dt \nonumber \\ 
&+ \integra \frac{\Phi^2 (\my,\zeta)}{2}  \rho (\mx-\my) \psi (\xi-\zeta)  d\nu_{(t,\my)}^2(\zeta) d\nu_{(t,\mx)}^1(\xi) d\my d\mx dt \nonumber \\ 
&+ \integra \rho(\mx-\my) \psi (\xi-\zeta) \ov{h^2}(t,\my,\zeta) \Phi(\mx,\xi) d\nu_{(t,\mx)}^1(\xi)d\zeta  d\my d\mx dW_t \nonumber \\  
&- \integra  \rho(\mx-\my) \psi (\xi-\zeta) h^1 (t,\mx,\xi)  \Phi(\my,\zeta) d\nu_{(t,\my)}^2(\zeta)   d\xi d\my d\mx dW_t
\nonumber \\
&-\integra \rho(\mx-\my) \psi(\xi-\zeta) \nu_{(t,\my)}^2(\zeta) d m_1(t,\mx,\xi) d\zeta d\my \nonumber \\ &-\integra  \rho(\mx-\my) \psi (\xi-\zeta) {\nu}_{(t,\mx)}^1(\xi) d m_2(t,\my,\zeta) d\xi d\mx.
\nonumber
\end{align}

Finally, moving the third term on the left hand side in \eqref{Main1} to the right hand side and using non-negativity of the measures $m_{1}$ and $m_2$ yields
\begin{align}
\label{Main2}
&\integral h^1 (T,\mx,\xi) \ov{h^2} (T,\my,\zeta) \rho(\mx-\my) \psi (\xi-\zeta) d\zeta d\xi d\my d\mx\\ &- \int\limits_{{\rm M}^2}\int\limits_{\R^2} h^1_0(\mx,\xi) \ov{h^2_0}(\my,\zeta) \rho(\mx-\my) \psi (\xi-\zeta)d\zeta d\xi d\my d\mx \nonumber \\ 
&+\integra \mff'(\mx,\xi) \cdot \nabla_{g,\mx} h^1 (t,\mx,\xi) \ov{h^2}(t,\my,\zeta) \rho(\mx-\my) \psi (\xi-\zeta) d\zeta d\xi d\my d\mx dt \nonumber \\ 
&+\integra \mff'(\my,\zeta) \cdot \nabla_{g,\my} \ov{h^2}(t,\my,\zeta) h^1 (t,\mx,\xi)  \rho(\mx-\my) \psi (\xi-\zeta) d\zeta d\xi d\my d\mx dt \nonumber
 \\  
&\leq \frac{1}{2} \integra (\Phi(\mx,\xi)-\Phi(\my,\zeta))^2 \rho (\mx-\my) \psi (\xi-\zeta) d\nu_{(t,\my)}^2(\zeta) d\nu_{(t,\mx)}^1(\xi) d\my d\mx  dt \nonumber \\ 
&+ \integra  \rho(\mx-\my) \psi (\xi-\zeta) \ov{h^2}(t,\my,\zeta)\Phi(\mx,\xi) d\nu_{(t,\mx)}^1(\xi)d\zeta d\my d\mx dW_t\nonumber \\ & - \integra  \rho(\mx-\my) \psi (\xi-\zeta) h^1 (t,\mx,\xi) \Phi(\my,\zeta) d\nu_{(t,\my)}^2(\zeta)    d\xi d\my d\mx dW_t. \nonumber
\end{align} Setting $\psi(\xi)=\delta(\xi)$ and $\rho(\mx)=\delta(\mx)$ and rearranging it a bit, we obtain 
\begin{align}
\label{Main5}
\nonumber
& \int\limits_\text{M} \int\limits_\mb{R} h^{1} (T,\mx,\xi) \ov{h^{2}} (T,\mx,\xi) d\xi d\mx \\ 
\nonumber
& \leq  \int\limits_\text{M} \int\limits_\mb{R} h^1_0 \ov{h^2_0} d\xi d\mx -\int\limits_0^\text{T} \int\limits_\text{M} \int\limits_\mb{R} \mff'(\mx,\xi) \cdot \nabla_{g,\mx} (h^1 (t,\mx,\xi) \ov{h^2}(t,\mx,\xi))  d\xi d\mx dt \\
& - \int\limits_0^\text{T} \int\limits_\text{M} \int\limits_\mb{R} \Phi(\mx,\xi) \p_\xi( h^1 (t,\mx,\xi) \ov{h^2}(t,\mx,\xi)) d\xi d\mx dW_t.
\end{align} 
Another integration by parts provides
\begin{align}
\label{Main5-1}
& \int\limits_\text{M} \int\limits_\mb{R} h^{1} (T,\mx,\xi) \ov{h^{2}} (T,\mx,\xi) d\xi d\mx 
\\
\nonumber
& \leq  \int\limits_\text{M} \int\limits_\mb{R} h^1_0 \ov{h^2_0} d\xi d\mx  + \int\limits_0^\text{T} \int\limits_\text{M} \int\limits_\mb{R} \Phi'(\mx,\xi) ( h^1 (t,\mx,\xi) \ov{h^2}(t,\mx,\xi)) d\xi d\mx dW(t)
\end{align} where we used the geometry compatibility conditions to eliminate the flux term.

By using non-negativity of $h^1$ and $\ov{h^2}$, we have after finding expectation of square of \eqref{Main5-1} and taking into account the It\^ o isometry 

\begin{align}
\label{Main5-2}
& E\left[\Big(\int\limits_\text{M} \int\limits_\mb{R} h^{1} (T,x,\xi) \ov{h^{2}} (T,\mx,\xi) d\xi d\mx \Big)^2 \right] \\ 
& \lesssim E \left[\Big( \int\limits_\text{M} \int\limits_\mb{R} h^1_0 \ov{h^2_0} d\xi d\mx  \Big)^2\right]  + \|\Phi'\|^2_{\infty} E\left[\int\limits_0^\text{T} \Big(\int\limits_\text{M} \int\limits_\mb{R} ( h^1 (t,\mx,\xi) \ov{h^2}(t,\mx,\xi)) d\xi d\mx \Big)^2 dt \right].
\nonumber
\end{align}

 From here, using the Gronwall inequality, we get
\begin{equation}
\label{gronwal}
E\left[\Big(\int\limits_\text{M} \int\limits_\mb{R} h^{1} (T,x,\xi) \ov{h^{2}} (T,\mx,\xi) d\xi d\mx\Big)^2 \right] \lesssim  E\left[\Big(\int\limits_\text{M} |u_{10}(\mx)-u_{20}(\mx)| d\mx \Big)^2 \right].
\end{equation}From here, if assume that $u_{10}=u_{20}$, we get almost surely for almost every $(t,\mx,\xi)\in [0,\infty)\times M \times \R$:
\begin{align*}
h^1(t,\mx,\xi)\, (1-h^2(t,\mx,\xi))=0.
\end{align*}

This implies that either $h^1(t,\mx,\xi)=0$ or $h^2(t,\mx,\xi)=1$.  Since we can interchange the roles of $h^1$ and $h^2$, we conclude that $1$ and $0$ are actually the only values that $h^1$ or $h^2$ can attain and that $h^1=h^2=h$. Since $h$ is also non-increasing with respect to $\xi$ on $[0,\infty)$, 
we conclude (taking into account the initial data $h_0={\rm sign}_+(u_0(\mx)-\xi)$) that there exists a function $u:[0,\infty)\times M \to \R$ such that 
\begin{equation}
\label{***} 
h(t,\mx,\xi)={\rm sign}_+(u(t,\mx)-\xi).
\end{equation}  We thus have the following corollary which is proven in the final section. 

\begin{corollary}
\label{Kor}
The stochastic kinetic solution to \eqref{main-eq}, \eqref{ic} has the form \eqref{***}. If the function $u$ satisfies the second item from Definition \ref{admissibility}, then it is an admissible stochastic solution to \eqref{main-eq}, \eqref{ic}.
\end{corollary}

\section{Uniqueness -- rigorous proof}
\label{uniquenessec}

In this section, we shall formalize the arguments from the previous section. To this end, it will be necessary to express \eqref{WeakFormFinal} in local coordinates. So, assume we are given a stochastic kinetic solution $h$. To prove uniqueness locally we take a chart $(U,\kappa)$ for $M$ and assume, without loss of generality, that $\kappa(U) = \R^d$. Define the local expression of $h$ as the map (in order to avoid proliferation of symbols, we shall keep the same notations for global and local quantities but we shall write $\tilde{\mx}$ to denote the local variable)
\[  h: \R^+ \times \R^d \times \R \times \Omega \to \R,\quad h(t,\tmx, \xi, \omega) = h ( t, \kappa^{-1}(\tmx), \xi, \omega) G(\tmx), \]
where ${G}(\tmx)$ is the Gramian corresponding to the chart $(U,\kappa)$. Similarly, for $\tmx \in \R^d$ we define
\begin{equation}
\label{conve}
\begin{split}
  {\Phi}(\tilde{\mx},\xi)&=\Phi(\kappa^{-1}(\tilde{\mx}),\xi),  \\
  {\mff}(\tmx,\xi)&=\mff(\kappa^{-1}(\tmx),\xi), \ \ {\mff}'(\tmx,\xi)=\mff'(\kappa^{-1}(\tmx),\xi) =a(\tmx,\xi) \\
  \nu_{(t,\tmx)}(\lambda)&=\nu_{(t,\kappa^{-1}(\tmx))}(\lambda)G(\tmx),
 \end{split}
\end{equation} and $m(t,\tmx,\xi)$ will be the pushforward measure of $m$ with respect to the mapping $\kappa$.

With such notations at hand, we now rewrite \eqref{WeakFormFinal} locally in the chart $(U,\kappa)$ into an equation in terms of $h_1(t,\tmx,\xi)$ and $\ov{h_2}(t,\tmx,\xi)$, which are two kinetic solutions to Cauchy problems corresponding to \eqref{main-eq} with the initial data $u_{10}$ and $u_{20}$, respectively. 
Below, we use the Einstein summation convention and we remind that $a=(a_1,\dots,a_d)=\mff'=(f_1',\dots,f_d')$. Also, since the equations are to be understood in the weak sense, we need to add the Gramian in each of the terms below except in $m_1$ and $m_2$, since the corresponding part in these termes is implied there by the definition of the pushforward measure. This is why we introduce the conventions from \eqref{conve}.

\begin{align}
\label{N1}
&d h^1(t,\tmx, \xi) +\diver_{\tmx} (a(\tmx, \xi)  h^1) dt + h^1 \Gamma^j_{kj}(\tmx) a_k(t,\tmx,\xi) dt \\&=  - \p_\xi \left(\frac{\Phi^2 (\tmx,\xi)}{2} \nu^1_{(t,\tmx)}(\xi) \right) dt + \Phi(\tmx,\xi) \nu^1_{(t,\tmx)}(\xi) dW_t +\p_\xi  dm_1 ,\nonumber\\
\label{N2}
&d \ov{h^2}(t,\tmy, \zeta)  +\diver_{\tmy} (a(\tmy, \zeta)   \ov{h^2}) dt +\ov{h^2} \Gamma^j_{kj}(\tmy)a_k(t,\tmy,\zeta) dt \\
\nonumber &=  \p_\zeta \left(\frac{\Phi^2 (\tmy,\zeta)}{2} \nu^2_{(t,\tmy)}(\zeta) \right) dt - \Phi(\tmy,\zeta)  \nu^2_{(t,\tmy)}(\zeta) dW_t -\p_\zeta dm_2
\end{align} 


We introduce two mollifying functions $\omega_1 \in \mathcal{D}(\mb{R}^d)$, $\omega_2 \in \mathcal{D}(\mb{R})$, where $d$ is the dimension of the manifold ${M}$, such that $\omega_i \geq 0$, $i=1,2$ and $\int_{\mb{R}^d} \omega_1 = \int_\mb{R} \omega_2=1$. Taking $\omega_{\delta, r}(\tmx, \xi)=\frac{1}{r \delta^d} \omega_1\left(\frac{\tmx}{\delta} \right)\omega_2\left(\frac{\xi}{r} \right)$, for some $\delta, r>0$, and using convolution, \eqref{N1} and \eqref{N2} yield (below and in the sequel, subscripts $\delta$ and $r$ denote convolution with respect to the corresponding variables):
\begin{align}
\label{conv}
& d h_{\delta,r}^1 + \diver_{\tmx} (a(\tmx,\xi) h_{\delta,r}^1) dt + g_{\delta, r}^1+\left( \Gamma^j_{kj}(\tmx) a_k(t,\tmx,\xi)h^1\right)_{\delta,r} dt \\
\nonumber
& = - \p_\xi \left(\frac{\Phi^2(\tmx, \xi)}{2} \nu_{(t,\tmx)}^1(\xi) dt \right)_{\delta, r} + (\Phi(\tmx,\xi) \nu_{(t,\tmx)}^1(\xi) dW_t)_{\delta, r} +\pa_\xi dm_{1,\delta,r} ,
\\
\label{conv1}
& d \ov{h_{\delta,r}^2} + \diver_{\tmy} (a(\tmy,\zeta) \ov{h_{\delta,r}^2}) dt + g_{\delta, r}^2+\left( \Gamma^j_{kj}(\tmy) a_k(t,\tmy,\zeta) \ov{h^2}\right)_{\delta,r} dt\\
\nonumber
& = \p_\zeta \left(\frac{\Phi^2(\tmy, \zeta)}{2} \nu_{(t,\tmy)}^1(\zeta) dt \right)_{\delta, r}  - (\Phi(\tmy,\zeta) \nu_{(t,\tmy)}^1(\zeta) dW_t)_{\delta, r}-\pa_\zeta dm_{2,\delta,r}
\end{align} where 
\begin{align*}
g_{\delta, r}^1&= \diver_{\tmx} (a(\tmx,\xi) h^1 dt )_{\delta, r} - \diver_{\tmx} (a(\tmx,\xi) h_{\delta,r}^1) dt\\
g_{\delta, r}^2&= \diver_{\tmy} (a(\tmy,\zeta) \ov{h^2} dt)_{\delta, r} - \diver_{\tmy} (a(\tmy,\zeta) \ov{h_{\delta,r}^2}) dt.
\end{align*}This term converges to zero as $\delta, r \to 0$ according to the Friedrichs lemma \cite{PR}. 

Now, multiplying \eqref{conv} and \eqref{conv1} with $\ov{h_{\delta,r}^2}=\ov{h_{\delta,r}^2}(t,\my,\zeta)$ and $h_{\delta,r}^1=h_{\delta,r}^1(t,\mx,\xi)$, respectively, and using \eqref{ItoFormulaFinal}, we obtain

\begin{align}
\label{MainConv}
 &d (h_{\delta,r}^1 \ov{h_{\delta,r}^2}) + 
(\Phi(\tmx,\xi) \nu_{(t,\tmx)}^1(\xi))_{\delta,r}(\Phi(\tmy,\mzt) \nu_{(t,\my)}^2(\mzt))_{\delta,r} dt  \\
 \nonumber &+\ov{h_{\delta,r}^2} \diver_{\tmx} (a(\tmx,\xi )h_{\delta,r}^1) dt + 
 h_{\delta,r}^1 \diver_{\tmy} (a(\tmy,\mzt)\ov{h_{\delta,r}^2}) dt \\&+ 
 \left( \Gamma^j_{kj}(\tmx) a_k(t,\tmx,\xi)h^1\right)_{\delta,r} \ov{h_{\delta,r}^2} dt +\left( \Gamma^j_{kj}(\tmy) a_k(t,\tmy,\zeta)\ov{h^2} \right)_{\delta,r} h_{\delta,r}^1 dt=
 \nonumber \\
 \nonumber & - 
 g^1_{\delta,r} \ov{h_{\delta,r}^2} -g^2_{\delta,r} h_{\delta,r}^1 +
\ov{h_{\delta,r}^2} (\Phi(\tmx,\xi) \nu_{(t,\tmx)}^1(\xi) dW_t)_{\delta, r} -  
h_{\delta,r}^1 (\Phi(\tmy,\mzt) \nu_{(t,\tmy)}^2(\mzt) dW_t)_{\delta, r}\\
& + h_{\delta,r}^1 \p_\mzt \left(\frac{\Phi^2(\tmy,\mzt)}{2} \nu_{(t,\tmy) dt }^2(\mzt)\right)_{\delta, r} - \ov{h_{\delta,r}^2} \p_\xi \left( \frac{\Phi^2(\tmx,\xi)}{2} \nu_{(t,\tmx) dt}^1(\xi) \right)_{\delta,r}
\nonumber\\
&+\ov{h_{\delta,r}^2} \p_\xi dm_{1,\delta,r}(t,\tmx,\xi) -h_{\delta,r}^1 \p_\zeta dm_{2,\delta,r}(t,\tmy,\zeta). 
\nonumber
\end{align}

Next, we choose non-negative functions $\rho \in \mathcal{D}(\mb{R}^d)$, $\psi, \varphi \in \mathcal{D}(\mb{R})$ such that $\int_{\mb{R}^d} \rho = \int_\mb{R} \psi=1$. Using the test function $\rho_\varepsilon (\tmx-\tmy) \psi_\varepsilon (\xi-\mzt) \varphi \left(\frac{\tmx+\tmy}{2} \right)$, with $\rho_\varepsilon (\tmx)=\frac{1}{\varepsilon^d} \rho\left( \frac{\tmx}{\varepsilon}\right)$, $\psi_\varepsilon (\xi)=\frac{1}{\varepsilon} \psi\left( \frac{\xi}{\varepsilon}\right)$, for some $\varepsilon>0$,  and integrating \eqref{MainConv} over $(0,T)$, the equation is rewritten in the variational formulation (recall that $h^1$ and $h^2$ are continuous with respect to $t\in \R^+$):

\begin{align}
\label{1-1}
 &\int\limits_{\R^{2d}} \int\limits_{\R^2} h_{\delta,r}^1(T,\tmx,\xi) \ov{h_{\delta,r}^2}(T,\tmy,\zeta)\rho_\varepsilon(\tmx-\tmy) \psi_\varepsilon (\xi-\mzt) \varphi \left( \frac{\tmx+\tmy}{2}\right) d\xi d\zeta  d\tmx d\tmy \\
 \label{1-2}
 &-\int\limits_{\R^{2d}} \int\limits_{\R^2} h_{0,\delta,r}^1(\tmx,\xi) \ov{h_{0,\delta,r}^2}(\tmy,\zeta)\rho_\varepsilon(\tmx-\tmy) \psi_\varepsilon (\xi-\mzt) \varphi \left( \frac{\tmx+\tmy}{2}\right) d\xi d\zeta  d\tmx d\tmy 
 \\
 \label{2-1}
 &+ 
  \int\limits_0^\text{T}\int\limits_{\R^{2d}} \int\limits_{\R^2} \big(\ov{h_{\delta,r}^2} \diver_{\tmx} (a(\tmx,\xi )h_{\delta,r}^1) + 
 h_{\delta,r}^1 \diver_{\tmy} (a(\tmy,\mzt)\ov{h_{\delta,r}^2}) \big) \times \\ & \qquad\qquad\qquad\qquad\qquad  \times\rho_\varepsilon(\tmx-\tmy) \psi_\varepsilon (\xi-\mzt) \varphi \left( \frac{\tmx+\tmy}{2}\right) d\xi d\zeta  d\tmx d\tmy dW_t \nonumber \\
 \label{2-2}
 & 
+\int\limits_0^\text{T}\int\limits_{\R^{2d}} \int\limits_{\R^2} \Big(\left( \Gamma^j_{kj}(\tmx) a_k(t,\tmx,\xi)h^1\right)_{\delta,r} \ov{h_{\delta,r}^2}+\left( \Gamma^j_{kj}(\tmy) a_k(t,\tmy,\zeta) \ov{h^2}\right)_{\delta,r} h_{\delta,r}^1\Big)\times \\ &\qquad\qquad\qquad \times  \rho_\varepsilon(\tmx-\tmy) \psi_\varepsilon (\xi-\mzt) \varphi \left( \frac{\tmx+\tmy}{2}\right) d\xi d\zeta  d\tmx d\tmy   \nonumber
  \\
  \label{3-0}
  & = - 
\int\limits_0^\text{T} \int\limits_{\R^{2d}} \int\limits_{\R^2} \Big(  g^1_{\delta,r} \ov{h_{\delta,r}^2}+g^2_{\delta,r} h_{\delta,r}^1 -
\ov{h_{\delta,r}^2} (\Phi(\tmx,\xi) \nu_{(t,\tmx)}^1(\xi) )_{\delta, r} +h_{\delta,r}^1 (\Phi(\tmy,\mzt) \nu_{(t,\tmy)}^2(\mzt) )_{\delta, r}\Big) \times \\& \nonumber \qquad\qquad\qquad \times\rho_\varepsilon(\tmx-\tmy) \psi_\varepsilon (\xi-\mzt) \varphi \left( \frac{\tmx+\tmy}{2}\right) d\xi d\zeta  d\tmx d\tmy dW_t \\
&\label{3-1}
+\int\limits_0^\text{T} \int\limits_{\R^{2d}} \int\limits_{\R^2}
\Big(h_{\delta,r}^1 \p_\mzt \left(\frac{\Phi^2(\tmy,\mzt)}{2} \nu_{(t,\tmy)}^2(\mzt)\right)_{\delta, r} - \ov{h_{\delta,r}^2} \p_\xi \left( \frac{\Phi^2(\tmx,\xi)}{2} \nu_{(t,\tmx)}^1(\xi) \right)_{\delta,r}\\ \nonumber & \qquad\qquad-(\Phi(\tmx,\xi) \nu_{(t,\tmx)}^1(\xi))_{\delta,r}(\Phi(\tmy,\mzt) \nu_{(t,\my)}^2(\mzt))_{\delta,r}  \Big)  \rho_\varepsilon(\tmx-\tmy) \psi_\varepsilon (\xi-\mzt) \varphi \left( \frac{\tmx+\tmy}{2}\right) d\xi d\zeta  d\tmx d\tmy  \\
\label{4-1}
&+\int\limits_0^\text{T} \int\limits_{\R^{2d}} \int\limits_{\R^2} \Big( \ov{h_{\delta,r}^2}(t,\tmy,\zeta) \p_\xi m_{1,\delta,r}(t,\tmx,\xi) - h_{\delta,r}^1(t,\tmy,\xi) \p_\zeta m_{2,\delta,r}(t,\tmy,\zeta) \Big)\times \\& \qquad\qquad\qquad\times\rho_\varepsilon(\tmx-\tmy) \psi_\varepsilon (\xi-\mzt) \varphi \left( \frac{\tmx+\tmy}{2}\right) d\xi d\zeta  d\tmx d\tmy. 
\nonumber
\end{align} We shall analyze this equality term by term. We start with the terms from \eqref{1-1}--\eqref{2-1}. We have:

\begin{align}
\label{LHS-1}
&\int\limits_{\R^{2d}} \int\limits_{\mb{R}^2} h_{\delta, r}^1 (T, \tmx, \xi) \ov{h_{\delta, r}^2} (T,\tmy,\mzt)  \rho_\varepsilon (\tmx-\tmy) \psi_\varepsilon (\xi-\mzt) \varphi \left(\frac{\tmx+\tmy}{2} \right) d\mzt d\xi d\tmy d\tmx  \\
\nonumber
&- \int\limits_{\R^{2d}} \int\limits_{\mb{R}^2} h_{\delta, r}^1 (0, \tmx, \xi) \ov{h_{\delta, r}^2} (0,\tmy,\mzt)  \rho_\varepsilon (\tmx-\tmy) \psi_\varepsilon (\xi-\mzt) \varphi \left(\frac{\tmx+\tmy}{2} \right) d\mzt d\xi d\tmy d\tmx  \\ 
\nonumber
& - \int\limits_0^\text{T} \int\limits_{\R^{2d}} \int\limits_{\mb{R}^2} a(\tmx, \xi) h_{\delta, r}^1 (t, \tmx, \xi) \ov{h_{\delta, r}^2} (t,\tmy,\mzt) \cdot \bigg [ \psi_\varepsilon (\xi-\mzt) \varphi \left(\frac{\tmx+\tmy}{2} \right) \nabla \rho_\varepsilon (\tmx-\tmy)  \\
\nonumber
& \qquad\qquad\qquad\qquad+ \frac{1}{2} \rho_\varepsilon (\tmx-\tmy) \psi_\varepsilon (\xi-\mzt) \nabla \varphi \left(\frac{\tmx+\tmy}{2} \right)  \bigg ] d\mzt d\xi d\tmy d\tmx dt  \\
\nonumber
& +\int\limits_0^\text{T} \int\limits_{\R^{2d}} \int\limits_{\mb{R}^2} a(\tmy, \mzt) h_{\delta, r}^1 (t, \tmx, \xi) \ov{h_{\delta, r}^2} (t,\tmy,\mzt) \cdot \bigg [ \psi_\varepsilon (\xi-\mzt) \varphi \left(\frac{\tmx+\tmy}{2} \right) \nabla \rho_\varepsilon (\tmx-\tmy)  \\
\nonumber
 & \qquad\qquad\qquad\qquad-\frac{1}{2} \rho_\varepsilon (\tmx-\tmy) \psi_\varepsilon (\xi-\mzt) \nabla \varphi \left(\frac{\tmx+\tmy}{2} \right)  \bigg ] d\mzt d\xi d\tmy d\tmx dt 
\\
\nonumber
& = \int\limits_{\R^{2d}} \int\limits_{\mb{R}^2} h_{\delta, r}^1 (T, \tmx, \xi) \ov{h_{\delta, r}^2} (T,\tmy,\mzt) \rho_\varepsilon (\tmx-\tmy) \psi_\varepsilon (\xi-\mzt) \varphi \left(\frac{\tmx+\tmy}{2} \right) d\mzt d\xi d\tmy d\tmx - \\
\nonumber
& - \int\limits_{\R^{2d}} \int\limits_{\mb{R}^2} h_{\delta, r}^1 (0, \tmx, \xi) \ov{h_{\delta, r}^2} (0,\tmy,\mzt) \rho_\varepsilon (\tmx-\tmy) \psi_\varepsilon (\xi-\mzt) \varphi \left(\frac{\tmx+\tmy}{2} \right) d\mzt d\xi d\tmy d\tmx  \\
\nonumber
& - \int\limits_0^\text{T} \int\limits_{\R^{2d}} \int\limits_{\mb{R}^2} (a(\tmx, \xi)-a(\tmy, \mzt)) \cdot  \nabla \rho_\varepsilon (\tmx-\tmy) h_{\delta, r}^1 (t, \tmx, \xi) \ov{h_{\delta, r}^2} (t,\tmy,\mzt)  \psi_\varepsilon (\xi-\mzt) \varphi \left(\frac{\tmx+\tmy}{2} \right) d\mzt d\xi d\tmy d\tmx dt  \\
\nonumber
& - \frac{1}{2} \int\limits_0^\text{T} \int\limits_{\R^{2d}} \int\limits_{\mb{R}^2} (a(\tmx, \xi)+a(\tmy, \mzt)) \cdot \nabla \varphi \left(\frac{\tmx+\tmy}{2} \right) h_{\delta, r}^1 (t, \tmx, \xi) \ov{h_{\delta, r}^2} (t,\tmy,\mzt)  \rho_\varepsilon (\tmx-\tmy) \psi_\varepsilon (\xi-\mzt) d\mzt d\xi d\tmy d\tmx dt
\end{align}

The penultimate term in \eqref{LHS-1} can be rewritten as (below $dV=d\mzt d\xi d\tmy d\tmx dt$):
\begin{align}
\label{Aux1}
& \int\limits_0^\text{T} \int\limits_{\R^{2d}} \int\limits_{\mb{R}^2} (a(\tmx, \xi)-a(\tmy, \mzt)) \cdot \nabla \rho_\varepsilon (\tmx-\tmy) h_{\delta, r}^1 (t, \tmx, \xi) \ov{h_{\delta, r}^2} (t,\tmy,\mzt)  \psi_\varepsilon (\xi-\mzt) \varphi \left(\frac{\tmx+\tmy}{2} \right)dV  = \\
\nonumber
& \int\limits_0^\text{T} \int\limits_{\R^{2d}} \int\limits_{\mb{R}^2} (a(\tmx, \xi)-a(\tmy, \mzt)) \cdot \nabla \left( \frac{1}{\varepsilon^d}\rho \left( \frac{\tmx-\tmy}{\varepsilon} \right)\right) h_{\delta, r}^1 (t, \tmx, \xi) \ov{h_{\delta, r}^2} (t,\tmy,\mzt)  \psi_\varepsilon (\xi-\mzt) \varphi \left(\frac{\tmx+\tmy}{2} \right) dV = \\
\nonumber
& \int\limits_0^\text{T} \int\limits_{\R^{2d}} \int\limits_{\mb{R}^2} (a(\tmx, \xi)-a(\tmy, \mzt)) \cdot \frac{1}{\varepsilon^{d+1}} \nabla \rho( \mz)\Big|_{\mz=\frac{\tmx-\tmy}{\varepsilon}} h_{\delta, r}^1 (t, \tmx, \xi) \ov{h_{\delta, r}^2} (t,\tmy,\mzt)  \psi_\varepsilon (\xi-\mzt) \varphi \left(\frac{\tmx+\tmy}{2} \right)  dV = \\
\nonumber
& \int\limits_0^\text{T} \int\limits_{\R^{2d}} \int\limits_{\mb{R}^2} \frac{a(\tmx, \xi)-a(\tmy, \mzt)}{\varepsilon} \cdot \frac{1}{\varepsilon^{d}} \nabla \rho ( \mz)\Big|_{\mz=\frac{\tmx-\tmy}{\varepsilon}} h_{\delta, r}^1 (t, \tmx, \xi) \ov{h_{\delta, r}^2} (t,\tmy,\mzt)  \psi_\varepsilon (\xi-\mzt) \varphi \left(\frac{\tmx+\tmy}{2} \right)  dV = \\
\nonumber
& \int\limits_0^\text{T} \int\limits_{\R^{2d}} \int\limits_{\mb{R}^2} \frac{a_k(\varepsilon \mz+\tmy, \xi)-a_k(\tmy, \mzt)}{\varepsilon z_k} z_k \pa_{z_k} \rho (\mz) h_{\delta, r}^1 (t, \tmy+\eps \mz, \xi) \ov{h_{\delta, r}^2} (t,\tmy,\mzt)  \psi_\varepsilon (\xi-\mzt) \varphi \left( \tmy + \frac{\varepsilon \mz}{2} \right) dV 
\nonumber
\end{align}where $\mz=\frac{\tmx-\tmy}{\varepsilon}$. We notice that, as $r,\delta,\varepsilon \to 0$ (in any order), this term becomes
\begin{align}
\label{term}
&\int\limits_0^\text{T} \int\limits_{\R^{d}} \int\limits_{\mb{R}} \pa_{\tilde{y}_k} a_k(\tmy,\xi)  h^1 (t, \tmy, \xi) \ov{h^2} (t,\tmy,\xi) \varphi ( \tmy ) \int_{\R^d} z_k \pa_{z_k}\rho (\mz) d\mz \, d\xi d\tmy  dt\\
&= -\int\limits_0^\text{T} \int\limits_{\R^{d}} \int\limits_{\mb{R}} \diver_{\tmy} a(\tmy,\xi)  h^1 (t, \tmy, \xi) \ov{h^2} (t,\tmy,\xi) \varphi ( \tmy ) d\xi d\tmy dt \nonumber\\
&\overset{\eqref{geomcomp}}{=} \int\limits_0^\text{T} \int\limits_{\R^{d}} \int\limits_{\R}  \Gamma^j_{kj}(\tmy) a_k(t,\tmy,\xi)h^1(t,\tmy,\xi) \ov{h_{\delta,r}^2}(t,\tmy,\xi) d\xi d\tmy dt.
\end{align} due to properties of the mollifier $\rho$. Thus, from \eqref{term} and \eqref{LHS-1} we conclude that as $r,\delta,\eps \to 0$ in any order
\begin{align}
&\eqref{1-1}+\eqref{1-2}+\eqref{2-1} \underset{r,\delta,\eps \to 0}{\longrightarrow}  \label{LHS-11} \\
&\int\limits_{\R^{d}} \int\limits_{\R} (h^1 \ov{h^2})(T,\tmy,\xi) \varphi(\tmx) d\xi d\tmx -\int\limits_{\R^{d}} \int\limits_{\R} (h_{0}^1\ov{h_{0}^2})(\tmx,\xi) \varphi(\tmx) d\xi  d\tmx  
 \nonumber\\&- \int\limits_0^\text{T} \int\limits_{\R^{d}} \int\limits_{\mb{R}} a(\tmx, \xi) (h^1  \ov{h^2}) (t, \tmx, \xi)  \nabla \varphi (\tmx)  d\xi  d\tmx dt-\int\limits_0^\text{T} \int\limits_{\R^{d}} \int\limits_{\R}  \Gamma^j_{kj}(\tmx) a_k(t,\tmx,\xi)h^1(t,\tmx,\xi) \ov{h_{\delta,r}^2}(t,\tmx,\xi) d\xi d\tmx dt.
 \nonumber
\end{align}

Term \eqref{2-2} is easy to handle. We simply let $r,\delta,\eps \to 0$ to conclude

\begin{align}
\eqref{2-2}  \underset{r,\delta,\eps \to 0}{\longrightarrow}  \int\limits_0^\text{T} \int\limits_{\R^{d}} \int\limits_{\mb{R}} 2 \Gamma^j_{kj}(\tmx) a_k(t,\tmx,\xi)h^1 \ov{h^2} \varphi(\tmx)  d\xi d\tmx dt. \label{LHS-2}
\end{align}

In order to prepare handling \eqref{3-0} and \eqref{3-1}, we use regularity of the function $\Phi$ (recall that $\Phi \in C^1_0(\R^d\times \R)$). We have 
\begin{align}
\label{Ph1}
&\int\limits_0^\text{T} \int\limits_{\R^{2d}} \int\limits_{\mb{R}^2}  
 \bigg [ \left( \frac{\Phi^2(\tmx,\xi)}{2} \nu_{(t,\tmx)}^1(\xi)  \right)_{\delta, r} \nu_{(t,\tmy), \delta, r}^2 (\mzt) -
 \frac{\Phi^2(\tmx,\xi)}{2} \nu_{(t,\tmx),\delta, r}^1(\xi)  \nu_{(t,\tmy), \delta, r}^2 (\mzt) 
 \bigg ] \times \\&\qquad\qquad\qquad \times \rho_\eps(\tmx-\tmy) \psi_\eps(\xi-\zeta) \varphi \Big( \frac{\tmx+\tmy}{2} \Big) d\mzt d\xi d\tmy d\tmx \underset{r,\delta \to 0}{\longrightarrow} 0,
 \nonumber
\end{align} and similarly
\begin{align}
\label{Ph2}
&\int\limits_0^\text{T} \int\limits_{\R^{2d}} \int\limits_{\mb{R}^2}  
 \bigg[\left( {\Phi(\tmx,\xi)} \nu_{(t,\tmx)}^1(\xi)  \right)_{\delta, r} \left( {\Phi(\tmy,\zeta)} \nu_{(t,\tmy)}^2(\zeta)  \right)_{\delta, r}- 
 \left( {\Phi(\tmx,\xi)} \nu_{(t,\tmx),{\delta, r}}^1(\xi)  \right)\, \left( {\Phi(\tmy,\zeta)} \nu_{(t,\tmy),{\delta, r}}^2(\zeta)  \right)
 \bigg]\times \\& \qquad\qquad\qquad \times \rho_\eps(\tmx-\tmy) \psi_\eps(\xi-\zeta) \varphi \Big( \frac{\tmx+\tmy}{2} \Big) d\mzt d\xi d\tmy d\tmx \underset{r,\delta \to 0}{\longrightarrow} 0.
 \nonumber
\end{align} In a similar fashion, we have

\begin{align*}
&\int\limits_0^\text{T} \int\limits_{\R^{2d}} \int\limits_{\mb{R}^2} \bigg[ \left( \Phi(\tmx,\xi) \nu_{(t,\tmx)}^1(\xi)  \right)_{\delta, r} \ov{h_{\delta, r}^2} (t, \tmy, \mzt) - \left(\Phi(\tmy,\mzt) \nu_{(t,\tmy)}^2(\mzt)  \right)_{\delta, r} h_{\delta, r}^1 (t, \tmx, \xi) \bigg] \times \\
& \qquad\qquad\qquad\times \rho_\varepsilon(\tmx-\tmy) \psi_\varepsilon (\xi-\mzt) \varphi \left( \frac{\tmx+\tmy}{2}\right) d\mzt d\xi d\tmy d\tmx dW_t
\nonumber
\\&= \int\limits_0^\text{T} \int\limits_{\R^{2d}} \int\limits_{\mb{R}^2} \bigg[ \Phi(\tmx,\xi) \nu_{(t,\tmx),\delta, r}^1(\xi)   \ov{h_{\delta, r}^2} (t, \tmy, \mzt) - \Phi(\tmy,\mzt) \nu_{(t,\tmy),\delta, r}^2(\mzt)   h_{\delta, r}^1 (t, \tmx, \xi) \bigg] \times \nonumber \\
& \qquad\qquad\qquad \times \rho_\varepsilon(\tmx-\tmy) \psi_\varepsilon (\xi-\mzt) \varphi \left( \frac{\tmx+\tmy}{2}\right) d\mzt d\xi d\tmy d\tmx dW_t+g_{3,\delta,r}
\nonumber
\end{align*} where $g_{3,\delta,r} \to 0$ as $\delta,r \to 0$. From here, using $\frac{d h^1_{\delta,r}(t,\tmx,\xi)}{\pa \xi} =-\nu_{(t,\tmx),\delta, r}^1(\xi)$ and $\frac{d \ov{h^2_{\delta,r}}(t,\tmy,\zeta)}{\pa \zeta} =\nu_{(t,\tmy),\delta, r}^2(\zeta)$ and integration by parts, we have the following conclusion for \eqref{3-0}
\begin{align}
\label{term2}
&\int\limits_0^\text{T} \int\limits_{\R^{2d}} \int\limits_{\mb{R}^2} \bigg[ \left( \Phi(\tmx,\xi) \nu_{(t,\tmx)}^1(\xi)  \right)_{\delta, r} \ov{h_{\delta, r}^2} (t, \tmy, \mzt) - \left(\Phi(\tmy,\mzt) \nu_{(t,\tmy)}^2(\mzt)  \right)_{\delta, r} h_{\delta, r}^1 (t, \tmx, \xi) \bigg] \times \\
& \qquad\qquad\qquad \times \rho_\varepsilon(\tmx-\tmy) \psi_\varepsilon (\xi-\mzt) \varphi \left( \frac{\tmx+\tmy}{2}\right) d\mzt d\xi d\tmy d\tmx dW_t
\nonumber\\&
\nonumber
\underset{{r,\delta,\eps \to 0}}{\longrightarrow} \int\limits_0^\text{T} \int\limits_{\R^{d}} \int\limits_{\R}  \Phi'(\tmx,\xi) h^1 (t,\tmx,\xi) \ov{h^2}(t,\tmx,\xi) \varphi \left( \tmx \right) d\xi d\tmx dW_t 
\nonumber
\end{align} where we used the procedure leading to \eqref{term}.

Having in mind \eqref{Ph1}, \eqref{Ph2}, and \eqref{term2}, we conclude that \eqref{3-1} has the following asymptotics: 
\begin{align}
\label{RHS}
&\int\limits_0^\text{T} \int\limits_{\R^{2d}} \int\limits_{\mb{R}^2} \left( \frac{\Phi^2(\tmx, \xi)}{2} \nu_{(t,\tmx)}^1(\xi)  \right)_{\delta, r} \ov{h_{\delta, r}^2} (t, \tmy, \mzt) \rho_\varepsilon(\tmx-\tmy) \psi_\varepsilon' (\xi-\mzt) \varphi \left( \frac{\tmx+\tmy}{2}\right) d\mzt d\xi d\tmy d\tmx dt + \\
\nonumber
& + \int\limits_0^\text{T} \int\limits_{\R^{2d}} \int\limits_{\mb{R}^2} \left( \frac{\Phi^2(\tmy,\mzt)}{2} \nu_{(t,\tmy)}^2(\mzt)  \right)_{\delta, r} h_{\delta, r}^1 (t, \tmx, \xi)  \rho_\varepsilon(\tmx-\tmy) \psi_\varepsilon' (\xi-\mzt) \varphi \left( \frac{\tmx+\tmy}{2}\right) d\mzt d\xi d\tmy d\tmx dt - \\
\nonumber
& - \int\limits_0^\text{T} \int\limits_{\R^{2d}} \int\limits_{\mb{R}^2} \left( \Phi(\tmx,\xi) \nu_{(t,\tmx)}^1(\xi)  \right)_{\delta, r}  \left(\Phi(\tmy,\mzt) \nu_{(t,\tmy)}^2(\mzt)  \right)_{\delta, r} \times \\
\nonumber
& \qquad\qquad\qquad\times \rho_\varepsilon(\tmx-\tmy) \psi_\varepsilon (\xi-\mzt) \varphi \left( \frac{\tmx+\tmy}{2}\right) d\mzt d\xi d\tmy d\tmx dt - \\
\nonumber
& - \int\limits_0^\text{T} \int\limits_{\R^{2d}} \int\limits_{\mb{R}^2} \left( \left( \Phi(\tmx,\xi) \nu_{(t,\tmx)}^1(\xi)  \right)_{\delta,r} \ov{h_{\delta, r}^2} (t, \tmy, \mzt) - \left(\Phi(\tmy,\mzt) \nu_{(t,\tmy)}^2(\mzt)  \right)_{\delta, r} h_{\delta, r}^1 (t, \tmx, \xi) \right) \times \\
\nonumber
& \qquad\qquad\qquad \times \rho_\varepsilon(\tmx-\tmy) \psi_\varepsilon (\xi-\mzt) \varphi \left( \frac{\tmx+\tmy}{2}\right) d\mzt d\xi d\tmy d\tmx dW_t \underset{r,\delta,\eps \to 0}{\longrightarrow} \\ \nonumber
&
\lim\limits_{\eps\to 0}\frac{1}{2} \int\limits_0^\text{T} \int\limits_{\R^{2d}} \int\limits_{\R^2} (\Phi(\tmx,\xi)-\Phi(\tmy,\zeta))^2 \rho_\eps(\tmx-\tmy) \psi_\eps (\xi-\zeta)\varphi \left( \frac{\tmx+\tmy}{2}\right) d\nu_{(t,\my)}^2(\zeta) d\nu_{(t,\mx)}^1(\xi) d\my d\mx dt\\ 
&\qquad \qquad -\int\limits_0^\text{T} \int\limits_{\R^{d}} \int\limits_{\R}  \Phi'(\tmx,\xi) \varphi(\tmx) h^1 (t,\tmx,\xi) \ov{h^2}(t,\tmx,\xi) d\xi d\tmx  dW_t
\nonumber\\
&=-\int\limits_0^\text{T} \int\limits_{\R^{d}} \int\limits_{\R}  \Phi'(\tmx,\xi)  \varphi(\tmx) h^1 (t,\tmx,\xi) \ov{h^2}(t,\tmx,\xi) d\xi d\tmx  dW_t.
\end{align} Finally, we want to get rid of the entropy defect measures from \eqref{4-1}. We use the fact that $h^1$ and $h^2$ are decreasing with respect to $\xi$ (i.e. $\zeta$) and that the measures $m_1$ and $m_2$ are non-negative. We have after two integration by parts (keep in mind that $\p_\xi \psi(\xi-\zeta)=-\pa_\zeta\psi(\xi-\zeta)$)
\begin{align}
\label{m1m2}
&\int\limits_0^\text{T} \int\limits_{\R^{2d}} \int\limits_{\R^2} \Big( \ov{h_{\delta,r}^2}(t,\tmy,\zeta) \p_\xi m_{1,\delta,r}(t,\tmx,\xi) -h_{\delta,r}^1(t,\tmy,\xi) \p_\zeta m_{2,\delta,r}(t,\tmy,\zeta) \Big) \times \\&\qquad\qquad\qquad \times \rho_\varepsilon(\tmx-\tmy) \psi_\varepsilon (\xi-\mzt) \varphi \left( \frac{\tmx+\tmy}{2}\right) d\xi d\zeta  d\tmx d\tmy  \nonumber
\\&=-\int\limits_0^\text{T}\int\limits_{\R^{2d}} \int\limits_{\R^2} \Big( \nu^2_{(t,\my),\eps,\delta}(\zeta)  m_{1,\delta,r}(t,\tmx,\xi)+\nu^1_{(t,\mx),\eps,\delta}(\xi)  m_{2,\delta,r}(t,\tmy,\zeta) \Big) \nonumber \times \\&\qquad\qquad\qquad \times \rho_\varepsilon(\tmx-\tmy) \psi_\varepsilon (\xi-\mzt) \varphi \left( \frac{\tmx+\tmy}{2}\right) d\xi d\zeta  d\tmx d\tmy \leq 0.
\nonumber
\end{align}

Finally, from \eqref{LHS-11}, \eqref{LHS-2}, \eqref{term2}, \eqref{RHS}, and \eqref{m1m2}, we conclude after letting $r,\delta, \eps \to 0$ (first $r,\delta\to 0$ and then $\eps\to 0$) that \eqref{1-1}--\eqref{4-1} becomes:

\begin{align*}
&\int\limits_{\R^d} \int\limits_\mb{R} h^{1} (T,\tmx,\xi) \ov{h^{2}} (T,\tmx,\xi) \varphi (\tmx)  d\xi d\tmx + \int\limits_0^\text{T} \int\limits_{\R^d} \int\limits_\mb{R} \Gamma^j_{kj}(\tmx) a_k(t,\tmx,\xi) h^{1} (t,\tmx,\xi) \ov{h^{2}} (t,\tmx,\xi) \varphi (\tmx)  d\xi d\tmx dt\\ 
&\leq  \int\limits_{\R^d} \int\limits_\mb{R} h^1_0 \ov{h^2_0} \varphi (\tmx)  d\xi d\tmx + \int\limits_0^\text{T} \int\limits_{\R^{d}} \int\limits_{\mb{R}} a(\tmx, \xi) \cdot \nabla \varphi (\tmx) (h^1  \ov{h^2}) (t, \tmx, \xi)   d\xi  d\tmx dt \\ 
&+ \int\limits_0^\text{T} \int\limits_{\R^d} \int\limits_\mb{R} \Phi'(\tmx,\xi) ( h^1 (t,\tmx,\xi) \ov{h^2}(t,\tmx,\xi)) \varphi (\tmx) d\xi d\tmx dW_t.
\end{align*} From here, using the definition of the integral over a manifold and recalling \eqref{conve}, we see that it holds

\begin{align}
\label{man-loc-est}
&\int\limits_\text{M} \int\limits_\mb{R} h^{1} (T,\mx,\xi) \ov{h^{2}} (T,\mx,\xi) G(\kappa(\mx)) \varphi (\mx)  d\xi d\mx \\ 
&\leq  \int\limits_\text{M} \int\limits_\mb{R} h^1_0(\mx,\xi) \ov{h^2_0}(\mx,\xi) G(\kappa(\mx)) \varphi (\mx)  d\xi d\mx  - \int\limits_0^\text{T} \int\limits_\text{M} \int\limits_{\mb{R}} (h^1  \ov{h^2}) (t, \mx, \xi) G(\kappa(\mx)) a(\mx, \xi) \cdot \nabla_g \varphi (\mx)  d\xi  d\mx dt \nonumber\\ 
&+ \int\limits_0^\text{T} \int\limits_\text{M} \int\limits_\mb{R} \Phi'(\mx,\xi) ( h^1 \ov{h^2})(t,\mx,\xi)G(\kappa(\mx))  \varphi (\mx) d\xi d\mx dW_t.
\nonumber
\end{align}Since we are on the compact manifold, we can take $\varphi\equiv 1$ which yields: 

\begin{align}
\label{man-loc-est-1}
&\int\limits_\text{M} \int\limits_\mb{R} h^{1} (T,\mx,\xi) \ov{h^{2}} (T,\mx,\xi) G(\kappa(\mx)) d\xi d\mx \\ 
&\leq  \int\limits_\text{M} \int\limits_\mb{R} h^1_0(\mx,\xi) \ov{h^2_0}(\mx,\xi)   G(\kappa(\mx)) d\xi d\mx  - \int\limits_0^\text{T} \int\limits_\text{M} \int\limits_{\mb{R}} (h^1  \ov{h^2}) (t, \mx, \xi) G(\kappa(\mx)) a(\mx, \xi) \cdot \nabla_g 1 \,    d\xi  d\mx dt \nonumber\\ 
&+\int\limits_0^\text{T} \int\limits_\text{M} \int\limits_\mb{R} \Phi'(\mx,\xi) ( h^1  \ov{h^2})(t,\mx,\xi) G(\kappa(\mx))  d\xi d\mx dW_t.
\nonumber
\end{align} We arrived to \eqref{Main5-1} plus a term which does not affect using the Gronwall inequality and It\^ o isometry which give uniqueness as in \eqref{Main5-2}. Remark that the Gramian has no influence on the procedure since it is a positive bounded function.

\section{Existence}\label{existencesec}

Our next aim is to prove that given initial data $u_0\in L^\infty(M)$, 
 there exists a stochastic kinetic solution $\chi$ in the sense of Definition \ref{kineticsol}, with the corresponding kinetic measure $m$. To this end, consider the vanishing viscosity approximation \eqref{SCL} augmented with initial conditions \eqref{ic}. We have the following theorem.
 
\begin{theorem}
\label{t1} 
For any $\eps >0$ the initial value problem \eqref{SCL}, \eqref{ic} with $u_0\in L^2(M)\cap L^\infty(M)$ has a  {stochastic solution} $u_\eps \in L^2([0,T); H^2(M))$. 
It satisfies, for any convex $\theta \in C^2(\R)$ such that $|\theta'(\lambda)| \leq C$, $\lambda \in \R$,

\begin{equation}
\label{app-ec}
\begin{split}
d\theta(u_\eps)\leq \Big(-\Div_g \int_0^{u_\eps} \theta'(v) \mff'(\mx,v) dv + \eps \Delta_g \theta(u_\eps)  +\frac{\Phi^2(\mx,u_\eps)}{2}\theta'' (u_\eps)\Big)dt&\\ +\Phi(\mx,u_\eps)\theta'(u_\eps)dW.&
\end{split}
\end{equation}

\end{theorem} 
\begin{proof}
We shall use the Galerkin approximation in order to first get approximate solutions to \eqref{SCL}, \eqref{ic}, and then we will prove that the sequence of approximate solutions converges along a subsequence toward a solution to \eqref{SCL}, \eqref{ic} satisfying \eqref{app-ec}. To this end, we fix an orthonormal  basis $\{e_k\}_{\N}$ in $L^2(M)$ consisted from eigenfunctions of the Laplace-Beltrami operator (see the subsection on the Riemannian manifolds).

We then look for the approximate solution in the form (after the next formula, we shall omit the stochastic variable to simplify the notation)
\begin{equation}
\label{un}
u_n=\sum\limits_{k=1}^n \alpha^n_k (t,\omega) e_k(\mx), \ \ t\in [0,T), \ \ \mx\in M, \ \ \omega\in \Omega.
\end{equation} We look for the functions $\alpha_k$, $k=1,\dots,n$, so that \eqref{SCL} is satisfied on the subspace of $L^2([0,T)\times \R^n)$ in the sense that almost surely it holds
\begin{equation}
\label{app-1}
\begin{split}
\int_M d u_n \varphi d\gamma(\mx)&= \int_M \mff(\mx,u_n) \nabla_g \varphi \, d\gamma(\mx) dt\\
&+\eps \int_M  u_n \, \Delta_g \varphi d\gamma(\mx) dt+\int_M \Phi (\mx,u_n) \varphi d\gamma(\mx) dW_t, \ \ \varphi\in Span\{e_k\}_{k=1,\dots,n}.
\end{split}
\end{equation} If we put here $\varphi=e_j$, $j=1,\dots, n$, using orthogonality of $\{e_k\}_{k\in \N}$, we get the following system of stochastic ODEs:
\begin{equation}
\label{system}
\begin{split}
d {\alpha}_j&=\int_M \mff(\mx,u_n) \nabla_g e_j(\mx) \, d\gamma(\mx) dt\\
&+\eps \lambda_j \alpha^n_j \int_M | e_j |^2 d\gamma(\mx) dt+\int_M \Phi (\mx,u_n) e_j(\mx) d\gamma(\mx) dW_t.
\end{split}
\end{equation} Since we assumed in \eqref{decay} that $\|\mff(\cdot,\lambda)\|_{L^\infty(M)} \leq C |\lambda|$ and $\Phi\in C^1_0(M \times \R) $, we know that \eqref{system} augmented with finite initial data has globally defined solution \cite{Oks}. In particular, we take
\begin{equation}
\label{id-fin}
\alpha_j(0,\omega)=\alpha_{j0}(\omega), 
\end{equation}for the coefficients $\alpha_{j0}$, $j\in \N$ of the initial data $u_0$ in the basis $\{e_k\}_{k\in \N}$:
$$
u_0(\mx,\omega)=\sum\limits_{k\in \N} \alpha_{j0}(\omega) e_j(\mx).
$$ Thus, we have obtained the sequence $(u_n)$ satisfying for every $n\in \N$ relation \eqref{system}. By taking $\varphi=u_n$ in \eqref{app-1} and using the procedure leading from \eqref{SCL} to \eqref{ISCL} with $\theta(u)=u^2/2$ to handle the stochastic part, we conclude:
\begin{align*}
\frac{1}{2} \int_M d |u_n|^2 d\gamma(\mx) &=\int_{M} \mff(\mx,u_n) \cdot \nabla_g u_n d\gamma(\mx)dt \\
&-{\eps}\int_M |\nabla_g u_n|^2 d\gamma(\mx)dt  +\int_M {\Phi^2(\mx,u_n)} d\gamma(\mx) dt+2\int_M \Phi(\mx,u_n) u_n d\gamma(\mx) dW_t.
\end{align*}Rewriting the latter expression in the integral form \eqref{integral}, squaring the expression, and using the Young inequality and the It\^{o} isometry, we get for a constant $C>0$:
\begin{equation}
\label{H1}
\begin{split}
&E\left[ \left(\int_M \int_0^T \frac{d |u_n|^2(t,\mx)}{2} dt d\gamma(\mx) \right)^2\right] +\eps^2 E\left[ \left(\int_0^T \int_M |\nabla u_n|^2 d\gamma(\mx) dt \right)^2 \right] \\
&\lesssim \Big(  E \left[ \left(\int_0^T\int_{M} \mff(\mx,u_n) \cdot \nabla_g u_n d\gamma(\mx) dt\right)^2 \right]
 +E \left[ \int_0^T\int_M \Phi^2(\mx,u_n) |u_n|^2(t,\mx) d\gamma(\mx) dt \right]\Big).
\end{split}
\end{equation} Now, since 
$$
E\left[ \left(\int_M \int_0^T \frac{d |u_n|^2(t,\mx)}{2} dt d\gamma(\mx) \right)^2\right]=E\left[ \left(\int_M \frac{|u_n|^2(T,\mx)}{2} d\gamma(\mx)-\int_M \frac{|u_0|^2(\mx)}{2} d\gamma(\mx) \right)^2 \right]
$$ and $\sup\limits_{\lambda \in \R} |\Phi^2(\mx,\lambda) |\lambda|^2 \in L^1(M)$ according to \eqref{assump-Phi}, we conclude from \eqref{H1}, using the Cauchy-Schwartz inequality and \eqref{decay}:
\begin{equation}
\label{H2}
\begin{split}
E\left[ \left(\int_M \frac{|u_n|^2(T,\mx)}{2}  d\gamma(\mx) \right)^2\right]+\eps E\left[\int_M |\nabla u_n|^2 d\gamma(\mx)\right] \leq C.
\end{split}
\end{equation}

Now, we need to estimate the $t$-variation of $(u_n)$. It is usual to find an $H^\alpha([0,T);H^{-s}(M))$-type estimate for some $s>0$ and then to interpolate with the obtained $L^2([0,T);H^1(M))$ estimate from the above. We proceed in this direction. First, take into account the integral formulation of the stochastic differential equation \eqref{app-1} given by \eqref{system}. We have in the weak sense in ${\cal D}'(M)\cap Span\{e_k\}_{k=1,\dots,n}$ (we omit the stochastic variable again):

\begin{align*}
&\int_M \left( u_n(t+\Delta t,\mx)-u_n(t,\mx) \right) \varphi(\mx)d\mx= 
\int_t^{t+\Delta t} \int_M \mff(\mx,u_n) \nabla_g \varphi(\mx) \, d\gamma(\mx) dt\\
&-\eps \int_t^{t+\Delta t} \int_M \nabla_g u_n \, \nabla_g \varphi(\mx) d\gamma(\mx) dt+\int_M \int_t^{t+\Delta t} \Phi (\mx,u_n) \varphi(\mx) d\mx dW_t, \ \ \varphi\in Span\{e_k\}_{k=1,\dots,n}.
\end{align*} Now, remark that if $\{e_k\}_{k\in \N}$ is an orthonormal basis in $L^2(M)$ then $\{\sqrt{\lambda_k} e_k\}_{k\in \N}$ is an orthonormal basis in $H^{-1}(M)$. This in turn implies that if $u_n=\sum\limits_{k=1}^n \alpha_k^n e_k(\mx)$ then 

\begin{equation}
\label{norm-rule}
\|u_n \|_{H^{-1}(M)} =\left( \sum\limits_{k=1}^n \frac{\left( \alpha_k^n \right)^2}{\lambda_k} \right)^{1/2}. 
\end{equation} Therefore, we choose above $\varphi(t,\mx)=e_k(\mx)/\sqrt{\lambda_k}$ to get:

\begin{align*}
&\frac{ \alpha^n_k(t+\Delta t)-\alpha^n_k(t)}{\sqrt{\lambda_k}}= -
\int_t^{t+\Delta t} \int_M \Div_g \mff(\mx,u_n)  \frac{e_k(\mx)}{\sqrt{\lambda_k}} \, d\gamma(\mx) dt'\\
&+\eps \int_t^{t+\Delta t} \int_M \Delta_g u_n \, \frac{ e_k(\mx)}{\sqrt{\lambda_k}} d\gamma(\mx) dt'+ \int_t^{t+\Delta t} \int_M \Phi (\mx,u_n) \frac{e_k(\mx)}{\sqrt{\lambda_k}} d\mx dW_{t'}, \ \ \varphi\in Span\{e_k\}_{k=1,\dots,n}.
\end{align*} We square the latter expression, find the expectation, and use the Cauchy-Schwartz and Jensen inequalities to infer:

\begin{align*}
&E\left[\frac{\left( \alpha^n_k(t+\Delta t)-\alpha^n_k(t) \right)^2}{{\lambda_k}} \right] \lesssim \Big( \Delta t E\left[ \int_t^{t+\Delta t}\left( \int_M \Div_g f(\mx,u_n) \frac{e_k(\mx)}{\sqrt{\lambda_k}} d\gamma(\mx) \right)^2 dt'\right] \\
&+ \eps^2 \Delta t E\left[ \int_t^{t+\Delta t} \left( \Delta_g u_n(t,\mx) \frac{e_k(\mx)}{\sqrt{\lambda_k}} \right)^2 dt'\right]+E\left[\left(\int_t^{t+\Delta t}\int_M \Phi(\mx,u_n)\frac{e_k}{\sqrt{\lambda_k}} d\gamma(\mx) dW_{t'} \right)^2\right]\Big).
\end{align*} We divide the expression by $\Delta t$, use here the Ito isometry and sum the expression over $k=1,\dots,n$. We have after taking into account \eqref{norm-rule}:
\begin{align*}
&E\big[\|u_n\|^2_{C^{1/2}([0,T);H^{-1}(M)}\big] \\
& \lesssim \Big( E\big[\int_t^{t+\Delta t}\|\Div_g f(\cdot, u_n(t',\cdot))\|^2_{H^{-1}(M)} dt' \big]+\eps^2  E\big[ \int_t^{t+\Delta t} \|\Delta_g u_n(\cdot,t')\|^2_{H^{-1}(M)}dt'\big]
\\&+ E \big[\frac{1}{\Delta t} \int_t^{t+\Delta t} \|\phi(\cdot,u_n(t,\cdot)) \|_{H^{-1}(M)} dt'\big]\Big),
\end{align*}and from here, since $E\big[\|\Delta_g u_n\|_{H^{-1}(M)}\big]=E\big[\| \nabla_g u_n \|_{L^2(M)}\big] \leq c <\infty$ according to \eqref{H2}, 
\begin{equation}
\label{H-t}
E\big[\|u_n\|^2_{C^{1/2}([0,T);H^{-1}(M)}\big] \leq c<\infty \ \ \implies \ \ E\big[\|u_n\|^2_{H^{1/2}([0,T);H^{-1}(M))}\big] \leq \tilde{c}<\infty.
\end{equation} Now, from \eqref{H2} and \eqref{H-t} and the interpolating between $L^2([0,T);H^{-1}(M))$ and $H^{1/2}([0,T);H^{-1}(M))$ we see that for any $s\in (0,1)$
$$
E\big[ \|u_n\|^2_{H^{s/2}([0,T);H^{1-s}(M)} \big] \leq C<\infty.
$$ From here, taking $s$ small enough, we conclude according to the Rellich theorem that $(u_n)$ is compact in $L^2([0,T)\times M)$ in the sense that there exists a stochastic function $u\in L^2([0,T)\times M)$ such that 
$$
E(\|u_n-u\|_{L^2([0,T)\times M)})\to 0 \ \ {\rm as} \ \ n\to \infty
$$ along a subsequence.

The function $u$ is a weak solution to \eqref{SCL}, \eqref{ic}. Since the equation \eqref{SCL} is locally strictly parabolic, it is also locally strictly parabolic and, using its local formulation, it is a standard issue (see e.g. \cite{Hof} for the stochastic situation) to conclude about $L^2([0,T);H^2(M))$-a.e. regularity of $u$. This immediately implies \eqref{app-ec-1} (see the derivation of \eqref{ISCL}; since the equation is linear with respect to $d u$, we do not need more regularity of $u$ with respect to $t$). \end{proof} 

From \eqref{app-ec}, it is not difficult to derive the kinetic formulation for \eqref{SCL}. Then, letting the approximation parameter $\eps\to 0$, we reach to the kinetic solution (Definition \ref{kineticsol}) to \eqref{main-eq}. Indeed, taking $\theta(u_\eps)=|u_\eps-\xi|_+$ in \eqref{app-ec} (as in Section 3) and remembering Schwatz lemma on non-negative distributions, we get for a non-negative measure stochastic $m_\eps$ after finding derivative with respect to $\xi$:
\begin{align}
\label{E1}
& d{\rm sign}_+(u_\eps-\xi)+\diver_g (\mff'(\mx, \xi)  {\rm sign}_+(u_\eps-\xi)) dt\\
&= \eps \Delta_g ({\rm sign}_+(u_\eps-\xi)) dt  -\p_\xi \left(\frac{\Phi^2 (\mx,\xi)}{2} \nu_\eps \right) dt + \Phi(\mx,\xi) \nu_\eps dW_t +\pa_\xi m_\eps,
\nonumber
\end{align} with $\nu_\eps=-\p_\xi {\rm sign}_+(u_\eps-\xi)$. Finally, taking a weak-$\star$ limit of $({\rm sign}_+(u_\eps-\xi))$ along a subsequence (denoted by $h$), we reach to \eqref{WeakFormFinal}. According to the standard procedure \cite{Dpe} presented in the proof of Theorem \ref{mainthm} below, we conclude that there exists $u$ such that $h(t,\mx,\xi)={\rm sign}_+(u-\xi)$ for a unique entropy admissible solution $u$ of \eqref{main-eq}, \eqref{ic}. Thus, we have the basic steps of the proof to the existence theorem. Before we prove it, we need the following simple lemma.

\begin{lemma}
\label{lconv-1}
Assume that  $E(\|u_\eps\|_{L^2(U)}) \leq C$, $U \subset\subset \R^+\times M$. 
Define
$$
U_n^l=\{ (t,\mx) \in U:\, E(|u_{n}(t,\mx,\cdot)|) > l \}.
$$ Then
\begin{equation}
\label{conv11}
\lim\limits_{l\to \infty} \sup\limits_{n\in \N} {\rm
meas}(U_n^l) =0.
\end{equation}

\end{lemma}
\begin{proof}
Since $(E(|u_{n}|^2))$ is bounded in and $U\subset\subset \R^+\times M$ it also holds $(E(|u_{n}|))$ is bounded. Thus, we have
\begin{align*}
&\sup\limits_{n\in \N} \int_{\Omega }E(|u_{n}(t,\mx,\cdot)|) dt d\mx \geq
\sup\limits_{n\in \N} \int_{\Omega_n^l} l dt d\mx\, \implies \\ & \implies \frac{1}{l}\sup\limits_{k\in \N} \int_{\Omega }|E(u_{n}(t,\mx))| d\mx \geq
\sup\limits_{n\in \N} {\rm meas}(\Omega_n^l),
\end{align*} implying \eqref{conv1}.
\end{proof}
Before we pass to the proof of the theorem, we need a notion of the truncation operator \cite{Dolz}
$$
T_N(z)=\begin{cases}
z, & |z|<N\\
N, &z\geq N\\
-N, & z\leq -N
\end{cases}.
$$ It is by now well known that if we can prove that if for the sequence $(u_n)$ bounded in $L^p$, $p>1$, the sequence of its truncation $(T_N(u_n))$ converges in $L^1_{loc}$ for every $N\in \N$, then the sequence itself converges in $L^1_{loc}$ as well \cite[]{Dolz}.

\begin{theorem}
\label{mainthm}
For any $u_0\in L^\infty(M)$ there exists a unique admissible stochastic solution to \eqref{main-eq}, \eqref{ic}.
\end{theorem}
\begin{proof}
First, let us prove that the family $(m_\eps)$ is the family of uniformly bounded functionals on $L_{loc}^2(\Omega;C_0([0,T]\times M \times [-R,R]))$ for any $R>0$. To this end, we simply take a test function $\varphi \in C^2_c([0,T]\times M\times [-R,R] )$ and test it against \eqref{E1}. We get

\begin{align*}
&\int\limits_{[0,\text{T}]\times \text{M} \times [-R,R]} \left(  -{\rm sign}_+(u_\eps-\xi)\pa_t \varphi-  {\rm sign}_+(u_\eps-\xi)) \, \mff'(\mx, \xi)\cdot \nabla_g \varphi \right) dt d\mx d\xi\\
&- \int\limits_{[0,\text{T}]\times \text{M} \times [-R,R]} \eps  ({\rm sign}_+(u_\eps-\xi)) \Delta_g \varphi dt d\mx d\xi  - \int\limits_{[0,\text{T}]\times \text{M} \times [-R,R]} \!\!\!\!\p_\xi \varphi(t,\mx,\xi) \, \frac{\Phi^2 (\mx,\xi)}{2} d\nu_\eps dt d\mx \\& - \int\limits_{[0,\text{T}]\times \text{M} \times [-R,R]} \varphi(t,\mx,\xi) \Phi(\mx,\xi) d\nu_\eps dW_t d\mx  = - \int\limits_{[0,\text{T}]\times \text{M} \times [-R,R]} \pa_\xi \varphi(t,\mx,\xi) d m_\eps ,
\nonumber
\end{align*} Finding square of the latter expression, using the basic Young inequality (sometimes called the Peter-Paul inequality) and the It{\^o} isometry, we get
\begin{align*}
&E \Big(\big|\int_{[0,\text{T}]\times \text{M} \times [-R,R]} \pa_\xi \varphi(t,\mx,\xi) d m_\eps \big|^2 \Big) \leq 5\Big( E\big(\big| \int_{[0,\text{T}]\times \text{M} \times [-R,R]}  |\pa_t \varphi| dt d\mx d\xi \big|^2 \Big) \\&+  E\Big(\big| \int_{[0,\text{T}]\times \text{M} \times [-R,R]} |\mff'(\mx, \xi)\cdot \nabla_g \varphi | dt d\mx d\xi \big|^2 \Big)  +E\Big(\big| \int_{[0,\text{T}]\times \text{M} \times [-R,R]}  \eps |\Delta_g  \varphi| dt d\mx d\xi \big|^2 \Big)\\&  + E\Big(\big{|}\sup\limits_{\xi\in [-R,R]}\int_{[0,\text{T}]\times \text{M}} \pa_\xi \varphi(t,\mx,\xi)\frac{\Phi^2 (\mx,\xi)}{2}dt d\mx \big|^2 \Big)+E\Big(\big{|}\sup\limits_{\xi\in [-R,R]}\int_{[0,\text{T}]} |\int_{\text{M}} |\varphi(t,\mx,\xi) \Phi(\mx,\xi)d\mx |  dt \big{|}^2)\Big).
\end{align*}
Here, we then choose $\int_{-R}^\xi \varphi(t,x,\eta) d\eta$ for a fixed $\varphi \in C^2_c([0,T]\times M;C^2([-R,R]))$ to get that the family $(m_\eps)$ is bounded functional on the Bochner space $L_{loc}^2(\Omega;C^2([0,T]\times M \times [-R,R]))$. Since $m_\eps$ are non-negative, according to the Schwartz lemma on non-negative distributions, we know that $m$ is bounded in If we fix $K\subset\subset M$, we know that $(m_\eps)$ is bounded functional on the Bochner space $L^2(\Omega;C_0([0,T]\times K \times [-R,R]))$. Thus, $(m_\eps)$ is weakly precompact in $L^2_{w\star}(\Omega;C_0([0,T]\times K \times [-R,R]))$ (see \cite[p.606]{Edv}). By using the Kantor diagonalization procedure, we conclude that there exists a subsequence $(m_n)_{n\in \N}$ of $(m_\eps)$ weakly converging toward $m\in L^2_{w\star}(\Omega;C_0([0,T]\times M \times [-R,R]))$. 

Now, we need to derive estimates for $u_\eps$. The procedure is essentially the same as when deriving the estimates for $(m_\eps)$. Indeed, take in \eqref{app-ec} the entropy $\theta(z)=z^2$ and integrate over $[0,T] \times M$. We get
\begin{align}
\label{app-ec-1}
& \int\limits_\text{M} u^2_\eps(T,\mx) d\gamma(\mx) - \int\limits_\text{M} u^2_0(T,\mx) d\gamma(\mx) \leq  \\
& \int\limits_0^\text{T} \int\limits_{\text{M}} \Phi^2(\mx,u_\eps) d\gamma(\mx)dt +\int\limits_0^\text{T} \int\limits_{\text{M}} 2 \Phi(\mx,u_\eps) u_\eps d\gamma(\mx) dW 
\nonumber
\end{align} where we used compactness of the manifold $M$ which provides
$$
\int\limits_\text{M} \Big(-\Div_g \int_0^{u_\eps} \theta'(v) \mff'(\mx,v) dv\Big) d\gamma(\mx)=0 \ \ {\rm and} \ \ \int\limits_\text{M} \Delta_gu_\eps d\gamma(\mx)=0.
$$ Now, we pass $\int_M u^2_0(T,\mx) d\gamma(\mx)$ to the right hand side, square both sides of such obtained expression and use the It\^o isometry in the last step to discover

\begin{align}
\label{app-ec-2}
& E\left( \left( \int\limits_\text{M} u^2_\eps(T,\mx) d\gamma(\mx) - \int\limits_\text{M} u^2_0(T,\mx) d\gamma(\mx) \right)^2 \right) \leq  \\
& 2 E\left( \left(\int\limits_0^\text{T} \int\limits_\text{M} \Phi^2(\mx,u_\eps) d\gamma(\mx)dt \right)^2 +\int\limits_0^\text{T} \left(\int\limits_\text{M} 2 \Phi(\mx,u_\eps) u_\eps d\gamma(\mx)\right)^2 dt \right) .
\nonumber 
\end{align} From here, using assumed bounds on the function $\Phi$, we conclude that the expected value of $L^2(M)$-norm of the sequence $(u_\eps(T,\cdot))$ is  bounded i.e., after integrating everything over $T\in [0,t]$, $t>0$, we concude that the expected value of $L^2([0,t]\times M)$-norm of the sequence $(u_\eps)$ is  bounded as well. 

Denote by $h$ weak-$\star$ limit in $L^\infty([0,T]\times M \times \R\times \Omega)$ along a subsequence of the family $(h_\eps)={\rm sign}_+(u_\eps-\lambda)$. It satisfies \eqref{WeakFormFinal} as well as Definition \ref{kineticsol}. Thus, it is a unique kinetic solution to \eqref{main-eq}, \eqref{ic}. According to Corollary \ref{Kor}, it has the form $h={\rm sign}_+(u-\xi)$. Finally, it remains to prove that the second item from the Definition \ref{admissibility} is satisfied. This follows from the fact that
$$
{\rm sign}_+(u_\eps(t,\mx)-\lambda)\underset{L^\infty-\star}{\rightharpoonup} {\rm sign}_+(u(t,\mx)-\lambda), \ \ \eps \to 0,
$$ from where, since ${\rm sign}(z)={\rm sign}_+(z)-(1-{\rm sign}_+(z))$ it also follows 
$$
{\rm sign}(u_\eps(t,\mx)-\lambda)\underset{L^\infty-\star}{\rightharpoonup} {\rm sign}(u(t,\mx)-\lambda), \ \ \eps \to 0.
$$

Fix $N>0$ and multiply this first by  the characteristic function of the interval $(-N,N)$ denoted by $\chi_{(-N,N)}(\lambda)$  and then by $\lambda \chi_{(-N,N)}(\lambda)$. We get for the truncation operator $T_N$:
\begin{equation}
\label{wconv}
\begin{split}
&T_N(u_\eps) \underset{{L^\infty-\star}}{\rightharpoonup} T_N(u) \ \ {\rm as} \ \ \eps\to 0\\
& (u_\eps^2-N^2 )\chi_{|u_\eps|<N}=T_N(u_\eps)^2-N^2 \underset{L^\infty-\star}{\rightharpoonup} (u_\eps^2-N^2 )\chi_{|u|<N}=T_N(u)^2-N^2  \ \ {\rm as} \ \ \eps\to 0.
\end{split}
\end{equation} Now, we consider for a fixed non-negative test function $\phi\in C_c([0,T]\times M)$:

\begin{align*}
&E\Big(\int\limits_{[0,\text{T}]\times \text{M}} \phi(t,\mx)\big( T_N(u_\eps)-T_N(u)\big)^2 dt d\gamma(\mx) \Big)\\
&=
E\Big(\int\limits_{[0,\text{T}]\times \text{M}} \phi(t,\mx)\big[ T_N(u_\eps)^2-T_N(u)^2 \\
& +2T_N(u) (T_N(u)-T_N(u_\eps))\big] dt d\gamma(\mx) \Big) \to 0
\end{align*} according to \eqref{wconv}. From here, we conclude that expectations of the truncations $(T_N(u_\eps))$ of the sequence $(u_\eps)$ strongly converge in $L_{loc}^2([0,T]\times M)$ toward \\ $E(T_N(u))=E(u^N)$.

From here, we can conclude about the convergence of $(E(|u_\eps -u|^2)$.
First, we  prove that the obtained sequence $(u^N)$  converges strongly in
$L^1_{loc}([0,T]\times M)$ as $N\to \infty$.

To this end, let $U\subset\subset [0,T]\times M$. It holds 
\begin{equation}
\label{uni_n}
\lim_l \sup_n E(\|{T_l(u_{n}) -  u_{n}}\|_{L^1(U)}) \to 0\,.
\end{equation}
Denote by
$$
U_n^l=\{ (t,\mx)\in U:\, E(u_{n}(t,\mx,\cdot)) > l \}.
$$
Since $(E(\|u^2_{n}\|_{L^2(U)}))$ is bounded, we have
\begin{align*}
& E\Big(\int_U |u_{n}-T_l(u_{n})| dt d\mx \Big)
\leq E\Big(\int_{U_n^l} |u_{n}|dt d\mx \Big) \underset{l\to \infty}  \leq {\rm meas}(U_n^l)^{1/2} \, E(\|u_n\|_{L^2([0,T]\times M)})\to 0
\end{align*} uniformly with respect to $n$ according to \eqref{conv1}.
Thus, \eqref{uni_n} is proved.

Next, we have
\begin{align}
\label{ul} E(\|u^{l_1}-u^{l_2}\|_{L^1( U)})
&\leq
E(\|u^{l_1}-T_{l_1}(u_{n_k})\|_{L^1( U)}+\|T_{l_1}(u_{n_k})-u_{n_k}\|_{L^1( U)})\nonumber\\
&+E(\|T_{l_2}(u_{n_k})-u_{n_k}\|_{L^1( U)}+\|T_{l_2}(u_{n_k})-u^{l_2}\|_{L^1( U)})\,,\nonumber
\end{align}
which together with  \eqref{uni_n} implies that $(u^l)$ is a Cauchy sequence with respect to expectation of the $L^1$-norm.  Thus, there exists a measurable function $u$ such that
\begin{equation}
\label{conv2}
E(\| u^l - u \|_{ L^1(U}))\to 0 \ \ {\rm as} \ \ l\to \infty.
\end{equation}

Now it is not difficult to see that entire $(u_{n_k})$ converges toward $u$
in the same norm as well as well. Namely, it holds

\begin{align*}
  E(\|u_{n_k}-u\|_{L^1(U)}) \leq
E(\|u_{n_k}-T_l(u_{n_k})\|_{L^1(U)}) + \\ +E(\|T_l(u_{n_k})-u^l\|_{L^1(U)})+E(\|u^l-u\|_{L^1(U)}),
\end{align*} which by the definition of functions $u^l$, and convergences \eqref{uni_n} and \eqref{conv2} imply the statement. Moreover, since $u_n$ is bounded in expectation of the $L^2$-norms, the function $u$ must be such as well. \end{proof}

\section{Acknowledgement}

The research is supported in part by the project P30233 of the Austrian Science Fund FWF. It is also supported by the Austria-Montenegro bilateral project "Stochastic flow over non-flat manifolds".

\end{document}